\def\ocirc#1{\ifmmode\setbox0=\hbox{$#1$}\dimen0=\ht0
    \advance\dimen0 by1pt\rlap{\hbox to\wd0{\hss\raise\dimen0
    \hbox{\hskip.2em$\scriptscriptstyle\circ$}\hss}}#1\else
    {\accent"17 #1}\fi}

\documentclass[a4paper,10pt]{amsart}
\usepackage{amsfonts}
\usepackage{amssymb}
\usepackage{amsmath}
\usepackage{newlfont}
\usepackage[mathcal]{euscript}
\usepackage{mathrsfs}
\usepackage{graphicx}
\usepackage{color}
\usepackage{enumerate}

 \newtheorem{thm}{Theorem}
\newtheorem{defn}[thm]{Definition}
\newtheorem{lem}[thm]{Lemma}
\newtheorem{cor}[thm]{Corollary}

\newtheorem{rem}[thm]{Remark}
\newtheorem{prop}[thm]{Proposition}

\DeclareMathOperator{\diam}{diam}
\DeclareMathOperator{\dist}{dist}
\DeclareMathOperator{\Per}{Per}

\DeclareMathOperator{\Tran}{Tran}
\DeclareMathOperator{\Rec}{Rec}

 \newcommand{\D}{\mathscr{D}}

 \newcommand{\orbp}{\text{Orb}^+}
 \newcommand{\orb}{\text{Orb}}
 \newcommand{\eps}{\varepsilon}
 
 \newcommand{\ra}{\rightarrow}

 \newcommand{\A}{\mathscr{A}}

 \newcommand{\Bl}{\mathcal{B}}

 \newcommand{\N}{\mathbb{N}}
 \newcommand{\Z}{\mathbb{Z}}

 \newcommand{\R}{\mathbb{R}}

 \newcommand{\abs}[1]{\left\vert#1\right\vert}
 \newcommand{\set}[1]{\left\{#1\right\}}

 \newcommand{\s}{\sigma}
\newcommand{\z}{\overline{z}}

\def\Arg{\operatorname{Arg}}
\def\cl{\operatorname{cl}}
\def\re{\mathfrak{R}}
\def\im{\mathfrak{I}}
\def\d{\Delta}

\newcommand{\picmm}[4]
{
\begin{figure}[!htbp]
\begin{center}
\includegraphics[width=#2]{#1}
\begin{minipage}[c]{0.8\textwidth}
\begin{center}
\caption{#3}\label{#4}
\end{center}
\end{minipage}
\end{center}
\end{figure}
}

\begin{document}


\title{Factor maps and invariant distributional chaos}

\author[M. Fory\'s]{Magdalena Fory\'s}
\address[M. Fory\'s]{Institute of Computer Science, Faculty of Mathematics and Computer Science, Jagiellonian University,
ul. \L ojasiewicza 6, 30-348 Krak\'ow, Poland}\email{magdalena.forys@uj.edu.pl}

\author[P. Oprocha]{Piotr Oprocha}
\address[P. Oprocha]{AGH University of Science and Technology\\
Faculty of Applied Mathematics\\
al. A. Mickiewicza 30, 30-059 Krak\'ow,
Poland\\ -- and --\\Institute of Mathematics\\ Polish Academy of Sciences\\ ul. \'Sniadeckich 8, 00-956 Warszawa, Poland} \email{oprocha@agh.edu.pl}

\author[P. Wilczy\'nski]{Pawe\l{} Wilczy\'nski}
\address[P. Wilczy\'nski]{Department of Applied Mathematics, Faculty of Mathematics and Computer Science, Jagiellonian University,
ul. \L ojasiewicza 6, 30-348 Krak\'ow, Poland}\email{pawel.wilczynski@uj.edu.pl}

\begin{abstract}
The main aim of this article is to show that maps with specification property have invariant distributionally scrambled sets and that this
kind of scrambled set can be transferred from factor to extension under finite-to-one factor maps. This solves some open questions in the literature
of the topic.

We also show how our method can be applied in practice,
taking as example Poincar\'{e} map of time-periodic nonautonomous planar differential equation
$$
\dot{z}=v(t,z)=\left(1+e^{i\kappa t}\abs{z}^2\right)\z^3 -N
$$
where $\kappa$ and $N$ are sufficiently small positive real numbers.
\end{abstract}



\maketitle
\section{Introduction}

The study of chaos in terms of dynamics of pairs dates back to 1975 to the famous paper by Li and Yorke \cite{LiYorke}.
During more than last 30 years, these ideas were widely extended. One of very natural extensions is distributional chaos,
introduced by Schweizer and Sm\'{\i}tal in \cite{SchwSm} as a condition equivalent to positive topological entropy
for interval maps. Presently it is known that there is no relation between existence of DC1 pairs (the strongest version of distributional chaos)
and positive topological entropy in general, which makes proving distributional chaos a more challenging task.

In \cite{OprochaDS} one of the authors of the present paper investigated properties of distributionally scrambled sets which are at the same
time invariant for the special case of interval maps. One of the natural questions which immediately arise is about conditions
sufficient for such a sets to appear in the dynamics. In fact, it was conjectured in \cite{OprochaDS} that specification property, together with some
other conditions can be a good candidate. In this paper, we prove in Theorem~\ref{thm:1}
that indeed it is the case. Except specification property itself, we need two additional assumptions, that is existence of a fixed point and a distal pair.
It is worth emphasizing that these two additional assumptions are the minimal conditions that have to be satisfied in our case. Simply, it is not hard to verify
that if $x$ and $f(x)$ are proximal (which is, in particular, the case here) then there is a fixed point in the system,
and it was proved in \cite{OprTAMS} that if a system does not have distal pair, then it cannot have DC1 pair as well.

Our second main result is Theorem~\ref{thm:myc-distr} which shows that invariant distributionally scrambled sets
can be transfered via a semi-conjugacy if this semi-conjugacy covers a periodic point (contained in our scrambled set)
finite-to-one. While this does not completely answers the question if distributional chaos transfers via finite-to-one semi-conjugacy from
distributionally chaotic system, it extends results obtained so far, e.g. in \cite{OpWil6}.
It is also nice completion to recently published results on invariant scrambled sets (e.g. see \cite{BlHuSn}).

As an application of obtained theorems
we prove in Section~\ref{sec:appl} that a Poincar\'{e} map induced by
time-periodic local process generated by the nonautonomous ordinary differential equation
$$
\dot{z}=v(t,z)=\left(1+e^{i\kappa t}\abs{z}^2\right)\z^3 -N
$$
is chaotic, where $N$ and $\kappa$ are small positive real numbers.
When $N$ goes from zero to a positive number
then stationary trajectory
is continued to three periodic trajectories (fixed point
of the Poincar\'{e} map explodes becoming at least 3 fixed points). This shows
real application of our Theorem~\ref{thm:1}, since previous methods on semi-conjugacy
and distributional chaos, e.g. these contained in \cite{OpWil5} or \cite{OpWil6} cannot be applied.
While we do not answer completely open question from \cite{OpWil6} mentioned before (see Remark~6 in \cite{OpWil6})
it provides a way around this problem, which works well with the method of isolating segments.
The authors strongly believe that this approach to proving distributional chaos can be successfully applied in many other situations
arising in applications.

\section{Preliminaries}

\subsection{Basic notation}
If $X$ is a topological space and $W$ is its subset, then by $\cl W$ or $\overline{W}$ we denote the closure of $W$ in $X$.

The diagonal in the product $X\times X$ is denoted $\Delta:=\set{(x,x)\; : \; x\in X}$ and $\Delta_\eps:=B(\Delta,\eps)$ for any $\eps>0$.
By a \emph{perfect set} we mean a nonempty compact set without isolated points and by a \emph{Cantor set} we mean a perfect and totally disconnected set (i.e. any connected subset of this set cannot have more than one point).
When a set $M$ can be presented as at most countable sum of Cantor sets, then we say that $M$ is a \emph{Mycielski set}.
If a set $A$ contains a countable intersection of open and dense subsets of $X$, then we say that $A$ is \emph{residual} in $X$.

The following fact is a simplified version of Mycielski theorem (see \cite{Akin}). It shows strong connections between Mycielski sets and residual relations.

\begin{thm}\label{thm:RelMycielski}
If $Z$ is perfect set in a compact metric space $(X,d)$ and if $R\subset Z\times Z$ is residual in $Z\times Z$ then there is a Mycielski set $M\subset Z$ such that $M$ is dense in $Z$ and $M\times M \subset R\cup \Delta$.
\end{thm}

\subsection{Topological dynamics}
In this paper \emph{dynamical system} $(X,f)$ always means a compact metric space (endowed with a metric $d$) together with a continuous map $f\colon X\ra X$.
Open balls are denoted by $B(x,\eps):=\set{y\in X\; : \; d(x,y)<\eps}$.

By $\orbp(x)$ we denote the set $\orbp(x):=\set{x,f(x),f^2(x),\dots}$ and
call it the \emph{(positive) orbit} of a point $x$.
If additionally $f$ is a homeomorphism, we may define its \emph{negative orbit} and \emph{(full)
orbit} by, respectively
$$
\orb^{-} (x,f)=\set{x, f^{-1}(x), f^{-2} (x),\dots},
\quad \quad \orb (x,f)=\orb^{-} (x,f)\cup \orbp (x,f).
$$
A point $y\in X$ is an \emph{$\omega$-limit
point} of a point $x$ if it is an accumulation point of the sequence
$x,f(x),f^2(x),\dots$. The set of all $\omega$-limit points of $x$
is said to be the \emph{$\omega$-limit set} of $x$ (or \emph{positive limit set} of $x$) and is denoted
$\omega(x,f)$. If $f$ is a homeomorphism, then we can define \emph{negative limit set of $x$} $\alpha(x,f)$, putting
$\alpha(x,f)=\omega(x,f^{-1})$.

We say that a point $x$ is \emph{periodic}
if $f^n(x)=x$ for some $n\geq 1$ and \emph{recurrent}
if $x\in \omega(x,f)$.  The set of all periodic points for $f$ is denoted by
$\Per(f)$ and by $\Rec(f)$ we denote the set of all recurrent points of $f$. By \emph{transitive point} we mean any point from the set $\Tran(f)=\set{x : \overline{\orbp(x,f)}=X}$.
A set $S$ is invariant for $f$ if $f(S)\subset S$ and \emph{strongly invariant} if $f(S)= S$.

The specification property was introduced by Bowen in
\cite{Bow3} (see also \cite{DenGrillSig}).
We say that $f$ has the \emph{periodic specification property} 
if, for any $\eps > 0$, there is an integer $N_\eps>0$
such that for any integer $s\geq 2$, any set $\set{y_1,\dots,y_s}$
of $s$ points of $X$, and any sequence $0=j_1\leq k_1 < j_2 \leq k_2
< \dots < j_s \leq k_s$ of $2s$ integers with $j_{l+1} - k_l\geq
N_\eps$ for $l= 1,\dots,s-1$, there is a point $x\in X$ such
that, for each $1\leq m\leq s$ and any $i$ with
$j_m \leq i \leq k_m$, we have:
\begin{eqnarray}
&&d(f^i(x),f^i(y_m))<\eps, \label{cond:psp1}\\
&&f^n (x)=x, \quad \textrm{ where } \;n=N_\eps+
k_s. \label{cond:psp2}
\end{eqnarray}
If we drop the periodicity condition \eqref{cond:psp2} from the above definition,
that is, if $f$ fulfills only the first condition above, then we say that $f$ has the
\emph{specification property}.

\begin{rem}\label{rem7}
Assume that $f$ has the specification property. Then by compactness we can use in the definition $s=\infty$, that is we can trace infinite sequence of points $\{y_i\}_{i \in \mathbb{N}} \subset X$ on during iterations within blocks $\{[j_m,k_m]\cap \N \}_{m \in \mathbb{N}}$ provided that $j_{m+1}-k_m\geq N_{\eps}$ for every $m$.
\end{rem}

\begin{rem}\label{rem8}
If $f$ is surjection with specification property then we can start tracing of orbit of each $y_m$ from $y_m$ itself instead of $f^{j_m}(y_m)$.
Strictly speaking, for every $\eps > 0$ the exists an integer $N_{\eps}>0$ such that for any integer $s\geq 2$, any sequence $\{y_1,\dots,y_s\} \subset X$ and any sequence of intervals $\{[j_m,k_m]\}_{m \in \mathbb{N}}$ with $j_{m+1}-k_m\geq N_{\eps}$, $m=1,\ldots,s$ there is a point $x \in X$ such that for each $m \in \mathbb{N}$ and $j_m \leq i \leq k_m$ the following condition holds:
$$
d(f^i(x),f^{i-j_m}(y_m))< \eps.
$$
\end{rem}

Let $(X,f)$, $(Y,g)$ be dynamical systems on compact metric spaces.
A continuous map $\Phi : X \ra Y$ is called a \emph{semiconjugacy} (or a \emph{factor map})
between $f$ and $g$ if $\Phi$ is surjective and $\Phi \circ f = g
\circ \Phi$. In the above case also say that $(Y,g)$ is a \emph{factor} of the system $(X,f)$.

\begin{prop}[\cite{OpWil6} Proposition 12]
\label{prop:hethom}
Let $(X,d), (Y,\rho)$ be compact metric spaces and $f\in C(X)$, $g\in C(Y)$.
Let $\Phi:X\ra Y$ be a semi-conjugacy such that $\Phi^{-1}(y)=\set{p_1,\ldots,p_k}\subset \Per(f)$ for some $y\in \Per(g)$ and $k\in \N$. Then, for any $z\in Y$ with the property $\omega_g(z)=\alpha_g(z)=\orb(y,g)$ and for any $q\in \Phi^{-1}(z)$ there exist $u,v\in \Phi^{-1}(y)$ such that $\alpha_f(q)=\orb(u,f)$ and $\omega_f(q)=\orb(v,f)$.
\end{prop}

\subsection{Scrambled sets and distributional chaos}

For any $A\subset \N$ define its \emph{upper density} $\D^*( A )$ and its \emph{lower density} $\D_*( A )$ by, respectively,
\begin{eqnarray*}
\D^*( A ) &=& \limsup_{n\ra \infty} \frac{\# ( A \cap \set{1, \ldots ,n})}{n},\\
\D_*( A ) &=& \liminf_{n\ra \infty} \frac{\# ( A \cap \set{1, \ldots ,n})}{n}.
\end{eqnarray*}

The concept of distributional chaos originated from \cite{SchwSm} is one of possible extensions of the definition of Li and Yorke \cite{LiYorke}.
It can be defined as follows.
For any positive integer $n$, points $x,y\in X$ and $t>0$, let
\begin{eqnarray*}
\Phi^{(n)}_{xy}(t) &=& \frac{1}{n}\# \set{i\;:\; d(f^{i}(x),f^{i}(y))<t \quad , 0\leq i
<n},\\
\Phi_{xy}(t)&=&\liminf_{n \ra \infty} \Phi^{(n)}_{xy}(t)\\
&=&\D_*\left(\set{i\; : \; d(f^i(x),f^i(y))<t} \right),\\
\Phi^*_{xy}(t)&=&\limsup_{n \ra \infty} \Phi^{(n)}_{xy}(t)\\
&=&\D^*\left(\set{i\; : \; d(f^i(x),f^i(y))<t} \right).
\end{eqnarray*}

\begin{defn}
If $x,y$ are such that $\Phi^*_{xy}(t)=1$ for all $t>0$ and $\Phi_{xy}(s)=0$ for some $s>0$ then we say that $(x,y)$ is a \emph{DC1 pair}.

We say that a set $S$ is \emph{distributionally scrambled} if every $(x,y)\in S\times S\setminus \Delta$ is a DC1 pair and \emph{distributionally $\eps$-scrambled} for some $\eps>0$ if  $\Phi_{xy}(\eps)=0$ for any distinct $x,y\in S$.

If there exists an uncountable distributionally scrambled set then we say that $f$ is \emph{distributionally chaotic (of type $1$)}, and if
additionally $S$ is distributionally $\eps$-scrambled then we say that distributional chaos is \emph{uniform}.
\end{defn}

\begin{rem}
A distributionally $\eps$-scrambled set $S$ can be closed or invariant, but never can have these two properties at the same time \cite{BlHuSn}.
\end{rem}

\subsection{Shift spaces}

Let $\A=\set{0,1,\ldots,n-1}$. We denote
$$
\Sigma_n=\A^\Z.
$$
By \emph{a word}, we mean any element of a free monoid $\A^*$ with the set of
generators equal to $\A$. If $x\in \Sigma_n$ and $i<j$ then by
$x_{[i,j]}$ we mean a sequence $x_i, x_{i+1}, \dots, x_j$. We may
naturally identify $x_{[i,j]}$ with the word $x_{[i,j]}=x_i
x_{i+1}\dots x_j\in \A^*$. It is also very convenient to denote $x_{[i,j)}=x_{[i,j-1]}$.

We introduce a metric $\rho$ in $\Sigma_n$ by
$$
\rho(x,y)=2^{-k}, \text{ where } k=\min\set{m\geq 0: x_{[-m,m]}\neq y_{[-m,m]}}.
$$

By the $0^\infty$ we denote the element $x\in\Sigma_n$ such that $x_i=0$ for all $i\in \Z$.
The usual map on $\Sigma_n$ is the shift map $\sigma$ defined by $\sigma(x)_i=x_{i+1}$ for all $i\in\Z$. The dynamical system
$(\Sigma_n,\sigma)$ is called the
\emph{full two-sided shift} over $n$ symbols.

If $X\subset \Sigma_n$ is closed and strongly invariant (i.e. $\sigma(X)= X$) then we say that $X$ is a shift. There are many equivalent ways to define shifts, e.g. $X$ is shift iff there exists a set (of forbidden words) $\mathcal{F}\subset \A^*$ such that $X=X_\mathcal{F}$ where
$$
X_\mathcal{F}=\set{x\in \Sigma_n \; : \; x_{[i,j]}\notin \mathcal{F} \textrm{ for every } i\leq j}.
$$

One of the most important classes of shifts is the class of shifts of finite type. It contains all shifts which can be defined by finite sets of
forbidden words. Equivalently, $X\subset \Sigma_n$ is a shift of finite type if there is an integer $m>0$ and $M\subset \A^m$ such that
$$
x\in X \quad \Longleftrightarrow \quad x_{[i,i+m)}\in M \textrm{ for all } i\in \Z
$$
A shift which may be obtained as a factor of a shift of finite type is called a \emph{sofic shift}.

Another way to define shifts of finite type
and sofic shifts is to use directed graphs and labeled directed graphs respectively, called their presentations (elements of shift are identified with bi-infinite paths on graph). The reader not familiar with this approach is once again referred to \cite{Kurka} or \cite{LindMarc}.

\subsection{Wa\.zewski method for local flows}

Let $X$ be a topological space and $W$ be its subset. The following definitions come from \cite{Srzed70}.
Let $D$ be an open subset of $\R\times X$.
By a {\it local flow} on $X$ we mean a continuous map $\phi:D\longrightarrow X$, such that three conditions are satisfied:
\begin{enumerate}[(i)]
\item $I_x=\{t\in \R:\; (t,x)\in D\}$ is an open interval $(\alpha_x,\omega_x)$ containing $0$, for every $ x\in X$,
\item $ \phi(0,x)=x$, for every $x\in X$,
\item $\phi(s+t,x)=\phi(t,\phi(s,x))$, for every $ x\in X$ and $s, t\in \R$ such that $s \in I_x$ and $ t\in I_{\phi(s,x)}$.
\end{enumerate}
In the sequel we write $\phi_t (x)$ instead of $\phi (t,x)$.
We distinguish three subsets of $W$ given by
\begin{align*}
W^-=&\{x\in W: \phi([0,t]\times\{x\})\not \subset W, \text{ for every }t>0\},\\
W^+=&\{x\in W: \phi([-t,0]\times\{x\})\not \subset W, \text{ for every }t>0\},\\
W^*=&\{x\in W: \phi(t,x)\not \in W, \text{ for some }t>0\}.
\end{align*}
It is easy to see that $W^-\subset W^*$.
We say that $W^-$ is the {\it exit set of $W$}, and $W^+$ the {\it entrance set of $W$}.
We say that $W$ is a {\it Wa\. zewski set} provided that:
\begin{enumerate}[(i)]
\item if $x\in W$, $t>0$, and  $\phi([0,t]\times\{x\})\subset \cl W$ then  $\phi([0,t]\times\{x\})\subset W$,
\item $W^-$ is closed relative to $W^*$.
\end{enumerate}
\begin{prop}
\label{prop:wazewski}
If both $W$ and $W^-$ are closed subsets of $X$ then $W$ is a Wa\. zewski set.
\end{prop}
The function $\sigma^*: W^*\longrightarrow [0,\infty)$
\begin{equation*}
\sigma^*(x)=\sup \{t\in [0,\infty): \phi([0,t]\times\{x\})\subset W\}
\end{equation*}
is called the {\it escape-time function of $W$}.
The following lemma is called the Wa\. zewski lemma.
\begin{lem}[{\cite[Lemma 2.1 (iii)]{Srzed70}}]
\label{lem:wazewski}
Let $W$ be a Wa\. zewski set and $\sigma^*$ be its escape-time function. Then $\sigma^*$ is continuous.
\end{lem}

\subsection{Processes and nonautonomous dynamics}
Let $X$ be a topological space and $\Omega \subset \mathbb{R} \times
\mathbb{R} \times X$ be an open set.

By a {\it local process} on $X$ we mean a continuous map $\varphi:
\Omega \longrightarrow X$, such that the following three conditions
are satisfied:
\begin{enumerate}[(i)]
\item $\forall \sigma\in \mathbb{R}$, $ x\in X$, $\{t\in
\mathbb{R}:\; (\sigma,t,x)\in \Omega\}$ is an open interval containing $0$,
\item $\forall \sigma \in \mathbb{R}$, $ \varphi(\sigma,0,\cdot)={\mathrm{id}} _X$,
\item $\forall x \in X$, $ \sigma, s \in \mathbb{R}$,  $ t\in \mathbb{R}$ if $ (\sigma,s,x)\in \Omega$,
$(\sigma+s,t,\varphi(\sigma,s,x))\in \Omega$  then $(\sigma,s+t,x)\in \Omega$ and $\varphi(\sigma,s+t,x)=\varphi(\sigma+s,t,\varphi(\sigma,s,x))$.
\end{enumerate}
For abbreviation, we write $\varphi_{(\sigma,t)}(x)$ instead of
$\varphi(\sigma,t,x)$.

Given a local process $\varphi$ on $X$ one can define a local flow $\phi$ on $\R\times X$ by
\begin{align*}
\phi(t,(\sigma,x))=(t+\sigma, \varphi(\sigma,t,x)).
\end{align*}
Using this identification, we can speak about Wa\.zewski sets for local processes.

Let $M$ be a smooth manifold and let $v: \mathbb{R} \times M \longrightarrow
TM$ be a time-dependent vector field. We assume that $v$ is regular
enough to guarantee that for every $(t_0,x_0)\in \mathbb{R}\times M$ the
Cauchy problem
\begin{align}
\label{1row} &\dot{x}= v(t,x),\\
\label{warp} &x(t_0)= x_0
\end{align}
has a unique solution. Then the equation \eqref{1row} generates a local
process $\varphi$ on $M$ by
$\varphi_{(t_0,t)}(x_0)=x(t_0,x_0,t+t_0)$, where $ x(t_0,x_0,\cdot)$ is the
solution of the Cauchy problem \eqref{1row}, \eqref{warp}.

Let $T$ be a positive number. We assume that $v$ is $T$-periodic
in $t$. It follows that the local process $\varphi$ is $T$-periodic,
i.e.,
\begin{align*}
\forall \sigma, t\in \mathbb{R}\;  \varphi_{(\sigma+T,t)}=\varphi_{(\sigma,t)},
\end{align*}
hence there is a one-to-one correspondence between $T$-periodic
solutions of (\ref{1row}) and fixed points of the Poincar\'{e} map
$P_T=\varphi_{(0,T)}$.

\subsection{Periodic isolating segments for processes}
\label{sec:pis}
Let $X$ be a topological space and $T$ be a positive number. We
assume that $\varphi$ is a $T$-periodic local process on $X$.

For any set $Z\subset \mathbb{R} \times X$ and $a,b,t \in \mathbb{R}$, $a<b$ we define
\begin{align*}
Z_t&=\{x\in X: (t,x)\in Z\}, \\
Z_{[a,b]}&=\{(t,x)\in Z: t\in [a,b]\}.
\end{align*}

Let $\pi_1: \mathbb{R} \times X \longrightarrow \mathbb{R}$ and $\pi_2: \mathbb{R} \times X
\longrightarrow X$ be projections on, time and space variable
respectively.

A compact set $W\subset [a,b]\times X$ is called an {\it isolating
segment over $[a,b]$ for $\varphi$} if it is ENR (Euclidean
neighborhood retract, e.g. see \cite{Dold})  and there are $W^{--},
W^{++} \subset W$ compact ENR's (called, respectively, the {\it
proper exit set} and {\it proper entrance set}) such that
\begin{enumerate}
\item $\partial W = W^- \cup W^+$,
\item $W^-=W^{--} \cup (\{b\}\times W_b)$ , $W^+=W^{++} \cup
(\{a\}\times W_a)$,
\item there exists homeomorphism $h:
[a,b]\times W_a \longrightarrow W$ such that $\pi_1 \circ h =
\pi_1$ and $h([a,b]\times W^{--}_a)=W^{--}$, $h([a,b]\times
W^{++}_a)=W^{++}$.
\end{enumerate}
Every isolating segment is also a Wa\. zewski set (for the local flow associated to a process $\varphi$). We say that an isolating segment $W$ over $[a,b]$ is {\it ($b-a$)-periodic} (or simply {\it periodic}) if $W_a = W_b$,
$W^{--}_a=W^{--}_b$ and $W^{++}_a=W^{++}_b$. Let $T>0$. Given the set $Z\subset [0,T]\times X$ such that $Z_0=Z_T$ we define its infinite concatenation by
$$
Z^\infty=\left\{(t,z)\in \R\times X: z\in Z_{t \hskip-5pt \mod T} \right\}.
$$

Let $W$ be a periodic isolating segment over $[a,b]$. The
homeomorphism $h$ induces $m: (W_a,W^{--}_a)\longrightarrow
(W_b,W^{--}_b)= (W_a,W^{--}_a) $ a {\it monodromy homeomorphism}
given by
\begin{align*}
 m(x)=\pi_2h(b,\pi_2h^{-1}(a,x)).
\end{align*}
A different choice of the homeomorphism $h$ leads to a map which
is homotopic to $m$. It follows that the automorphism in singular
homology
\begin{align*}
\mu_W = H(m): H(W_a,W^{--}_a)\longrightarrow H(W_a,W^{--}_a)
\end{align*}
is an invariant of segment $W$.

The following theorem, proved by R. Srzednicki in \cite{Srzed10}, play
the crucial role in the method of isolating segments.
\begin{thm}\cite[Theorem 5.1]{SrzedWojZgli}
\label{tws}

Let $\varphi$ be a local process on $X$ and let $W$ be a
periodic isolating segment over $[a,b]$. Then the set
\begin{align*}
U=U_W=\{x\in W_a: \varphi_{(a,t-a)}(x)\in W_{t}\setminus W_t^{--}\;
\forall t\in [a,b]\}
\end{align*}
is open in $W_a$ and the set of fixed points of the restriction
\begin{equation*}
\varphi_{(a,b-a)}|_U:U\longrightarrow W_a
\end{equation*}
is compact. Moreover
\begin{align*}
{\mathrm{ind}} (\varphi_{(a,b-a)}|_U)={\mathrm{Lef}}(\mu_W)
\end{align*}
where ${\mathrm{Lef}}(\mu_W)$ denotes the Lefschetz number of $\mu_W$.

In particular, if
\begin{equation*}
{\mathrm{Lef}}(\mu_W) \neq 0
\end{equation*}
then $\varphi_{(a,b-a)}$ has a fixed point in $W_a$.
\end{thm}

For the definition of the fixed point index and Lefschetz
number the reader is referred to \cite{Dold}.

\subsection{Continuation method}
Let $X$ be a metric space. We denote by $\rho$ the corresponding distance on $\R\times X$. Let $\varphi$ be a local process on $X$, $T>0$ and $W, U$ be two subsets of $\R\times X$. We consider the following conditions (see \cite{SrzedWojZgli, WojZgli4}):
\begin{itemize}
\item[(G1)] $W$ and $U$ are $T$-periodic segments for $\varphi$ which satisfy
\begin{equation}
\label{war:G1}
U\subset W, \quad (U_0,U_0^{--})=(W_0,W_0^{--}),
\end{equation}
\item[(G2)] there exists $\eta>0$ such that for every $(t,w)\in W^{--}$ and $(t,z)\in U^{--}$ there exists $\tau_0>0$ such that for $0<\tau<\tau_0$ holds $(t+\tau,\varphi(t,\tau,w))\not \in W$, $\rho((t+\tau_0,\varphi_{(t,\tau_0)}(w)),W)>\eta$ and $(t+\tau,\varphi(t,\tau,z))\not \in U$, $\rho((t+\tau_0,\varphi_{(t,\tau_0)}(z)),U)>\eta$.
\end{itemize}
Let $K$ be a positive integer and let $E[1], \ldots, E[K]$ be disjoint closed subsets of the essential exit set $U^{--}$ which are $T$-periodic, i.e. $E[l]_0=E[l]_T$, and such that
\begin{equation*}
 U^{--}=\bigcup_{l=1}^K E[l].
\end{equation*}
(In applications we will use the decomposition of $U^{--}$ into connected components).

Before we can recall the method of continuation, we need one more definition related to the set $W\subset \R\times X$.
For $n\in \N$, $D\subset W_0$ and every finite sequence $c=(c_0,\ldots, c_{n-1})\in \{0,1,\ldots, K\}^{\{0,1,\ldots,n-1 \}}$ we define $D_c$ as a set of points satisfying the following conditions:
\begin{itemize}
\item[(H1)] $\varphi_{(0,lT)}(x)\in D$ for $l\in \{0,1,\ldots,n\}$,
\item[(H2)] $\varphi_{(0,lT+t)}(x)\in W_t\setminus W_t^{--}$ for $t\in [0,T]$ and $l\in \{0,1,\ldots,n-1\}$,
\item[(H3)] for each $l=0,1,\ldots, n-1$, if $c_l=0$, then $\varphi_{(0,lT+t)}(x)\in U_t\setminus U_t^{--}$ for $t\in (0,T)$,
\item[(H4)] for each $l=0,1,\ldots, n-1$, if $c_l>0$, then $\varphi_{(0,lT)}(x)$ leaves $U$ in time less than $T$ through $E[c_l]$.
\end{itemize}

Let $\Omega\subset \R\times \R\times X$ be open and
\begin{equation*}
[0,1]\times \Omega\ni (\lambda, \sigma,t,x)\mapsto \varphi^\lambda_{(\sigma,t)}(x)\in X
\end{equation*}
be a continuous family of $T$-periodic local processes on $X$. We say that the conditions (G1) and (G2) are satisfied {\it uniformly} (with respect to $\lambda$) if they are satisfied with $\varphi$ replaced by $\varphi^\lambda$ and the same $\eta$ in (G2) is valid for all $\lambda\in [0,1]$.

We write $D_c^\lambda$ for the set defined by the conditions (H1)--(H4) for the local process $\varphi^\lambda$.

The following theorem plays the crucial role in the method of continuation.
\begin{thm}[see \cite{SrzedWojZgli, WojZgli1}]
 \label{thm:cont}
Let $\varphi^\lambda$ be a continuous family of $T$-periodic local processes such that (G1) and (G2) hold uniformly. Then for every $n>0$ and every finite sequence $c=(c_0,\ldots, c_{n-1})\in \{0,1,\ldots, K\}^{\{0,1,\ldots,n-1 \} }$ the fixed point indices ${\mathrm{ind}}\left(\varphi^\lambda_{(0,nT)}\mid_{(W_0\setminus W_0^{--})_c^\lambda}\right)$ are correctly defined and equal each to the other (i.e. do not depend on $\lambda \in [0,1]$).
\end{thm}

\section{Finite to one factor maps}

\begin{thm}\label{thm:1}
Let $(X,f)$, $(Y,g)$ be dynamical systems, where $Y$ is an infinite set, let $y\in \Tran(g)\cap \Rec(g)$ and let $p\in Y$ be a fixed point of $g$.
If $\pi\colon (X,f)\ra (Y,g)$ is such that $\pi^{-1}(p)$ is finite and consists only of fixed points for $f$, then there is a perfect and $f$-invariant set $Z\subset X$ such that $\pi(Z)=Y$ and
\begin{enumerate}
\item\label{thm:1:c1} if $\Phi_{py}^*(\alpha)=1$ for every $\alpha>0$ then there is a residual subset $R\subset Z\times Z$ such that $\Delta\subset R$ and $\Phi^*_{f^m(u),f^n(v)}(s)=1$ and $\Phi^*_{f^m(u),f^n(u)}(s)=1$ for every $s>0$, every $(u,v)\in R$ and every $m,n\geq 0$,
\item\label{thm:1:c2} if $\Phi_{f^n(y)y}(\beta)=0$ for some $\beta>0$ and all $n>0$ then there is $\gamma$ and a residual subset $Q\subset Z\times Z\setminus \Delta$ such that for any $(u,v)\in Q$ we have:
\begin{enumerate}[(a)]
\item\label{thm:1:c2a} $\Phi_{f^m(u),f^n(v)}(\gamma)=0$ for every  $m,n\geq 0$,
\item\label{thm:1:c2b} $\Phi_{f^m(u),f^n(u)}(\gamma)=0$ for every  $m\neq n$, $m,n>0$.
\end{enumerate}
\end{enumerate}
\end{thm}
\begin{proof}
The first step is to define set $Z$.
To do so, consider the family $\Bl$ of closed and $f$-invariant subsets of $X$ such that $\pi(A)=Y$ for any $A\in \Bl$. We claim that any chain
in the partial ordering induced by sets inclusion has a lower bound. Simply, let $A_0 \supset A_1 \supset \ldots$, where each $A_i\in \Bl$.
If we put $A_\infty \bigcap_{i=0}^\infty A_i$ then it is closed, nonempty, and is also invariant since
$$
f(A_\infty)=f(\bigcap_{i=0}^\infty A_i)\subset \bigcap_{i=0}^\infty f(A_i)\subset \bigcap_{i=0}^\infty A_i=A_\infty.
$$
Additionally $\pi^{-1}(y)\cap A_i\neq \emptyset$ and so for every $i$ there is $y_i\in A_i$ such that $\pi(y_i)=y$. But since $\set{y_j}_{j=i}^\infty\subset A_i$, without loss of generality we may assume that $y^*=\lim_{i\ra \infty}y_i$ is well defined. Note that $y^*\in A_\infty\cap \pi^{-1}(y)$ and therefore $A_\infty \in \Bl$ showing that $A_\infty$ is the claimed lower bound.

Let $Z$ be a minimal element in $\Bl$ provided by Kuratowski-Zorn lemma and let $z\in \pi^{-1}(y)\cap Z$. We claim that $z\in \omega(z,f)$. Simply, if we fix any $q\in Y$ then there is a sequence $n_i$ such that $q=\lim_{i\ra \infty} f^{n_i}(y)$ and additionally, going to a subsequence if necessary, we may assume that $q'=\lim_{i\ra \infty} f^{n_i}(z)$ is also well defined. Then
$$
q=\lim_{i\ra \infty} f^{n_i}(y)=\lim_{i\ra \infty} f^{n_i}(\pi(z))=\pi(\lim_{i\ra \infty} f^{n_i}(z))=\pi(q')
$$
and thus $\pi(\omega(z,f))=Y$. But $Z$ is minimal element in $\Bl$, in particular $f(Z)=Z$, so since $z\in Z$, we have $\omega(z,f)\subset Z$ and $\omega(z,f)\in \Bl$ which immediately implies that $\omega(z,f)=Z$.
If $Z$ is a finite set then $Y$ is finite as well, which is a contradiction. Then the only possibility is that $Z$ is infinite, and so it is
a perfect set, since it is an $\omega$-limit set.

Now we are ready to prove \eqref{thm:1:c1}.
Let us fix any integer $l>0$ and denote by $R_l^{m,n}$ the set of these pairs $(q,r)\in Z\times Z$ such that $\Phi_{f^m(q)f^n(r)}^{(k)}(1/l)>1-1/l$ for some $k>l$. It is obvious that $R_l$ is
open in $Z$ by uniform continuity of $f$.

Fix any nonempty open sets $U,V\subset Y$. There are $m'<n'$ such that $f^{m'}(z)\in U$, $f^{n'}(z)\in V$. Without loss of generality we may assume that $m\leq n$ and denote $K=n+n'$.
Denote $\pi^{-1}(\set{p})=\set{u_0,\ldots, u_s}$ and let $\eps=\min\set{d(u_i,u_j)/4 \; : \; i\neq j}$ and
assume additionally that $\eps<1/2l$.
Let $0<\delta<\eps$ be such that $\diam (f^i(A))<\eps$ for $i=0,1,\ldots, k$ provided that $\diam (A)<\delta$. There is $\eta>0$ such that
$$
\pi^{-1}(B(p,\eta)) \subset \bigcup_{j=0}^s B(u_i,\delta).
$$

Since $D^*\left(\set{i\; : \; d(f^i(y),f^i(p))<\eta} \right)=1$ it is easy to verify that we also have that
$$
D^*\left(\set{i\; : \; d(f^j(y),p)<\eta \textrm{ for }j=i,\ldots,i+K} \right)=1.
$$
By the above observation, we immediately get $k>l$ such that
$$
\frac{1}{k} \# \set{m'+m\leq i <m'+m+k \; : \; d(f^j(y),p)<\eta \textrm{ for }j=i,\ldots,i+K}> 1-\frac{1}{l}
$$
Note that if $d(f^{i+j}(y),p)<\eta$ for $j=0,\ldots,K$ then for any $j$ there is $\phi(j)$ such that $d(f^{i+j}(z),u_{\phi(j)})<\delta$.
But then $f(f^{i+j+1}(z),u_{\phi(j)})<\eps$ which implies $\phi(j)=\phi(j+1)$ since otherwise $d(u_{\phi(j)},u_{\phi(j+1)})>2\eps$ which is impossible.
In particular, if $d(f^{i+j}(y),p)<\eta$ for $j=0,\ldots,K$ then
$$
d(f^{i}(z),f^{i+n-m+n'-m'}(z))<2\eps<1/l.
$$
By the above we calculate that
\begin{eqnarray*}
1-\frac{1}{l} &\leq & \frac{1}{k} \# \set{0\leq i <k \; : \; d(f^{m'+m+i+j}(y),p)<\eta \textrm{ for }j=0,\ldots,K}\\
&\leq& \frac{1}{k} \# \set{0\leq i <k \; : \; d(f^{m+i}(f^{m'}(z)),f^{n+i}(f^{n'}(z)))<1/l}.
\end{eqnarray*}
This shows that $d(f^{m'}(z),f^{n'}(z))\in R_l^{m,n}$ thus $R_l^{m,n} \cap U\times V\neq \emptyset$. Indeed, $R_l^{m,n}$ is open and dense for any $m,n\geq 0$ and $l>0$.

Now consider sets $W_l^m\subset Z$ consisting of points $q$ such that $\Phi_{f^m(q)q}^{(k)}(1/l)>1-1/l$ for some $k>0$. Then, by exactly the same arguments as before we get that $W_l^m$ is open and dense in $Z$. Observe that the set $(R_l^{m,n}\cup \Delta) \cap (W_l^m \times W_l^m)$ is open and dense in $Z\times Z$ and $(u,v)\in R_l^{m,n}$ if and only if $(v,u)\in R_l^{n,m}$.
Thus the set
$$
R=\bigcap_{l>0}\bigcap_{m,n\geq 0} (R_l^{m,n}\cup \Delta) \cap (W_l^m \times W_l^m)\supset \Delta
$$
is residual and symmetric subset of $Z\times Z$. Note that for any $(u,v)\in R$ and any $m,n>0$ we have two possibilities. If $u\neq v$ then $(u,v)\in R_l^{m,n}$ for any $l\geq 0$. In particular, for any fixed $\xi>0$ there is are increasing sequences $l_i,k_i$ such that $1/l_i<\xi$ and
$$
\Phi^{(k_i)}_{f^m(u),f^n(v)}(\xi) \geq \Phi^{(k_i)}_{f^m(u),f^n(v)}(1/l_i)\geq 1-1/l_i \longrightarrow 1
$$
which gives $\Phi^{*}_{f^m(u),f^n(v)}(\xi)=1$.
If $u=v$ then $u\in W_l^{|m-n|}$ and thus $\Phi^{*}_{f^{|m-n|}(u),v}(\xi)=1$ and $\Phi^{*}_{u,f^{|m-n|}(v)}(\xi)=1$ for any $\xi>0$. From this we immediately
get that also $\Phi^{*}_{f^m(u),f^n(v)}(\xi)=1$ which ends the proof of \eqref{thm:1:c1}.

Next we are going to prove \eqref{thm:1:c2}, that is, we are going to construct an appropriate set $Q$. Let $\gamma>0$ be such that if $d(u,v)\leq \gamma$ then $d(\pi(u),\pi(v))<\beta$.
Note that if $(f^i(z),f^{n+i}(z))\in \Delta_{\gamma/2}$, then $d(f^i(z),f^{i+n}(z))<\gamma$ and so $d(g^i(x),g^{i+n}(x))<\beta$.
In other words
$$
\frac{1}{k} \#\set{0\leq i<k\; : \; d(f^i(z),f^{i}(f^n(z)))\geq \gamma}\geq 1- \Phi_{xg^n(x)}^{(k)}(\beta).
$$
Now, if we denote by $Q_l^{(m,n)}$ the set of these $(u,v)\in Z\times Z$ such that for some $k>l$ we have:
$$
\frac{1}{k} \#\set{0\leq i<k \; : \; d(f^{i}(f^m(u)),f^{i}(f^n(v)))\geq \gamma}\geq 1-1/l
$$
then, repeating arguments similar to the previous proof that $R_l^{m,n}$ is open and dense, we easily obtain that $Q_l^{(m,n)}$  is also open and dense.
Similarly, for $m\neq n$, $m,n\geq 0$, the set $\widetilde{Q}_l^{(m,n)}$ the set of these $(u,v)\in Z\times Z$ such that for some $k>l$ we have:
$$
\frac{1}{k} \#\set{0\leq i<k \; : \; d(f^{i}(f^m(u)),f^{i}(f^n(u)))\geq \gamma}\geq 1-1/l
$$
is open and dense.
But then the set
$$
Q=(\bigcap_{l>0}\bigcap_{m,n\geq 0} Q_l^{m,n}\cap Q_l^{n,m}) \cap (\bigcap_{l>0}\bigcap_{m\neq n \text{ and } m,n\geq 0} \widetilde{Q}_l^{m,n}\cap \widetilde{Q}_l^{n,m})
$$
is residual in $Z\times Z$ and $Q\cap \Delta=\emptyset$.
Now, observe that if $m,n\geq 0$ and $(u,v)\in Q \subset Q_l^{m,n}$
then
\begin{eqnarray*}
\Phi_{f^m(u),f^n(v)}(\gamma)&\leq & \frac{1}{k} \#\set{0\leq i<k \; : \; d(f^{i}(f^m(u)),f^{i}(f^n(v)))< \gamma}\\
&\leq & \lim_{l\to \infty}\frac{1}{l} =0.
\end{eqnarray*}
Similarly, if $m\neq n$ and $m,n\geq 0$ then condition $(u,v)\in Q \subset \widetilde{Q}_l^{m,n}\cap \widetilde{Q}_l^{m,n}$
immediately implies that
\begin{eqnarray*}
\Phi_{f^m(u),f^n(u)}(\gamma)&\leq & \frac{1}{k} \#\set{0\leq i<k \; : \; d(f^{i}(f^m(u)),f^{i}(f^n(u)))< \gamma}\\
&\leq & \lim_{l\to \infty}\frac{1}{l} =0,\\
\Phi_{f^m(u),f^n(u)}(\gamma)&\leq & \frac{1}{k} \#\set{0\leq i<k \; : \; d(f^{i}(f^m(u)),f^{i}(f^n(u)))< \gamma}\\
&\leq & \lim_{l\to \infty}\frac{1}{l} =0.
\end{eqnarray*}
The proof is completed.
\end{proof}

\begin{cor}\label{cor:InvDC}
Let $(X,f)$, $(Y,g)$ be dynamical systems, let $y\in \Tran(g)\cap \Rec(g)$, let $p\in Y$ be a fixed point of $g$ and let $(y,p)$ form a DC1 pair.
If $\pi\colon (X,f)\ra (Y,g)$ is such that $\pi^{-1}(p)$ is finite and consists only of fixed points for $f$. Then for some $\eps>0$ there is a Mycielski invariant distributionally $\eps$-scrambled set $M$ for $f$.
\end{cor}
\begin{proof}
First note that since $(p,y)$ is a DC1 pair, the set $Y$ is infinite.
Let $Z\subset X$ be a perfect set and let $R$ and $Q$ be relations provided by Theorem~\ref{thm:1}. By Theorem~\ref{thm:RelMycielski} there is a Mycielski set $C\subset Z$
such that $C\times C \subset (R\cap Q)\cup \Delta$.
First, we claim that $C$ can contain at most one eventually periodic point. Simply, if $(u,v)\not\in R$ and $u$ is eventually periodic then $f^n(u)$ must be a fixed point for some $n$ (different points on a cycle are never proximal). If $u\neq v$ and both $u,v$ are eventually periodic then $f^n(u)=f^n(v)$ for some $n$ or $f^n(u)\neq f^n(v)$ for every $n$. But the case $f^n(u)=f^n(v)$ is impossible, since $(u,v)\in Q$. and in the second case, for some $n$ both $f^n(u),f^n(v)$ are fixed points, and since these points are different $(u,v)\not\in R$, which also leads to a contradiction. Indeed $C$ contains at most one eventually periodic point. Thus without loss of generality, we may assume that $C$ does not contain eventually periodic points (if we remove a finite set $A$ from a Cantor set $K$ then we can easily find a disjoint Cantor sets in $K$ disjoint with $A$ but with sum dense in $K$). Note that by the definition of $Q$ the set $C$ is distributionally scrambled, therefore $f^n|_C$ injective map for any $n=1,2,\ldots$, in particular $f^n(C)$ is a Mycielski set for every $n$.

Denote $M=\bigcup_{i=0}^\infty f^i(C)$ and observe that it is a Mycielski set as a countable union of Mycielski set. Obviously $f(M)\subset M$ (and it is dense in $Z$ as well), so it is enough to show that $M$ is distributionally $\eps$-scrambled with $\eps=\gamma/2$ where $\gamma$ was provided by Theorem~\ref{thm:1} for $Q$. To do so, fix any $(u,v)\in M$. There are $n,m\geq 0$ and $x,y\in C$ such that $f^n(x)=u$ and $f^m(y)=v$.
If we can choose $x,y$ in a way that $x\neq y$ then by the definition of relations $Q$ and $R$, in particular by the condition \eqref{thm:1:c2a} in Theorem~\ref{thm:1}, we have that
\begin{eqnarray*}
&&\Phi^*_{uv}(\xi)=\Phi^*_{f^n(x),f^m(y)}(\xi)=1 \quad\quad \text{ for every }\xi>0, \text{ and }\\
&&\Phi_{uv}(\eps)=\Phi^*_{f^n(x),f^m(y)}(\eps)=0.
\end{eqnarray*}
The second possibility is $x=y$, and then simply by the definition of $M$ we must have $m\neq n$. But since $x\in C$ and $C$ is infinite, there exists at least one
point $q\in C$, $q\neq x$ and so $(x,q)\in Q\cap R$.
But then again it is guaranteed by condition \eqref{thm:1:c2b} in the definition of $Q$ in
Theorem~\ref{thm:1} that
\begin{eqnarray*}
&&\Phi^*_{uv}(\xi)=\Phi^*_{f^n(x),f^m(x)}(\xi)=1 \quad\quad \text{ for every }\xi>0, \text{ and }\\
&&\Phi_{uv}(\eps)=\Phi^*_{f^n(x),f^m(x)}(\eps)=0.
\end{eqnarray*}
The proof is finished.
\end{proof}

\section{Specification property and distributional chaos}

The aim of this section is to show that systems with specification property contain invariant dystributionally $\eps$-scrambled sets.
This answers Question~1 from \cite{OprochaDS}, which was left open for further research.

We start our considerations with a few auxiliary lemmas, which show that there is a special type of separation of orbits in maps with specification.

\begin{lem}\label{lem13}
Let $(X,d)$ be a compact metric space with $\#X>1$ and let $f\colon X \ra X$ be a surjective continuous map with specification property.
Then for every $n>0$ there exists $z=z(n)$ such that $\inf_{i\geq 0} d(f^{i}(z),f^{i+n}(z))>0$.
\end{lem}
\begin{proof}
Fix any two distinct points $p,q\in X$, fix $n>0$ and put $\eps := d(p,q)$. We apply the specification property for $\frac{\eps}{3}$, points  $p,q$ and the sequence of times $0=j_1=k_1<j_2=k_2<\dots$, where $k_{m+1}-j_m=N$.
Strictly speaking, by the specification property with help of Remarks~~\ref{rem7} and \ref{rem8}, we can find a point $z \in X$ such that for every $i=0,1,2,\ldots$ we have:
\begin{align*}
d(f^{2Ni}(z),p)&<\frac{\eps}{3},\\
d(f^{(2i+1)N}(z),q)&<\frac{\eps}{3}.
\end{align*}
Without loss of generality, increasing $N$ when necessary, we may assume that $N$ is a multiple of $n$.
Note that by the above and the choice of points $p,q$, we have that
\begin{eqnarray}
&d(f^{kN}(z),f^{k(N+1)}(z))>\frac{\eps}{3}>0\quad\quad \text{ for every } k\geq 0\label{k_to_N_eps}.
\end{eqnarray}

Since $f$ is uniformly continuous, we can find $\delta>0$ such that for any $u\in X$ the following implication holds:
$$
d(u,f^{n}(u))<\delta \quad\Longrightarrow\quad \left[ \; d(f^{i}(u),f^{i+n}(u)) < \frac{\eps}{3N} \quad\text{ for }i=0,1,\ldots, 2N \; \right].
$$
But there is $s>0$ such $N=sn$ and so $d(u,f^{n}(u))<\delta$ implies that
\begin{eqnarray}
d(u,f^{N}(u)) &\leq & \sum_{i=0}^{s-1} d(f^{in}(u),f^{in+n}(u))\nonumber\\
&\leq & \sum_{i=0}^{s-1} \frac{\eps}{3N} < \frac{\eps}{3}.\label{n_to_N_1}
\end{eqnarray}
To finish the proof it is enough to show that $\inf_{i\geq0}d(f^i(z),f^{i+n}(z))>0$, so assume on the contrary
that $\inf_{i\geq0}d(f^i(z),f^{i+n}(z))=0$.
By the above, using uniform continuity of $f$, there is $k>0$ such that $d(f^{k+i}(z),f^{k+i+n}(z))<\delta$ for $i=0,1,\ldots, N$ we can find
$m$ such that $m$ is a multiple of $N$ and $d(f^{m}(z),f^{m+n}(z))<\delta$. This, by \eqref{n_to_N_1} immediately gives that for some $k>0$ we have that
$$
d(f^{kN}(z),f^{(k+1)N}(z))=d(f^{m}(z),f^{m+N}(z))<\frac{\eps}{3}
$$
which is in contradiction with \eqref{k_to_N_eps}. The proof is finished.
\end{proof}

\begin{lem}\label{lem:dist_n}
Let $(X,d)$ be a compact metric space with $\#X>1$ and let $f\colon X \ra X$ be a surjective continuous map with specification property.
There is $\eps>0$ and a~sequence $\set{z_n}_{n=1}^\infty \subset X$ so that
$\inf_{i\geq 0} d(f^{i}(z_n),f^{i+n}(z_n))\geq \eps$ for every $n$.
\end{lem}
\begin{proof}
Lemma~\ref{lem13} ensures that for every $n=1,2,\ldots $ there exists $y_n \in X$ such that $\inf_{i\geq 0} d(f^i(y_n),f^{i+n}(y_n))= \delta_n>0$.
Denote $p=y_1$, $q=f(y_1)$ and $\eta = \delta_1=\inf_{i\geq 0} d(f^i(p),f^i(q))$
Let $N$ be provided by the specification property for points $p,q$ and $\eta/3$, which enables us, similarly as in the proof of Lemma~\ref{lem13}
to find for every $n\geq 0$ a point $z_{n+N}$ such that
\begin{eqnarray}
d(f^i(p),f^{i+2j(n+N)}(z_{n+N}))& <&\frac{\eta}{3} \quad\text{ for } 0\leq i \leq n \text{ and } j\in \N\nonumber\\
d(f^i(q),f^{i+(2j+1)(n+N)}(z_{n+N}))& <&\frac{\eta}{3} \quad\text{ for } N\leq i\leq n+N \text{ and } j\in \N \label{c:iv*}
\end{eqnarray}
By the above, points $z_m$ are defined for every $m>N$.
For $m\leq N$ we put $z_m=y_m$. There exists $\delta_0<\frac{\eta}{3}$ such that if $d(x,y)<\delta_0$ then $d(f^i(x),f^i(y))<\eta/3$ for $i=0,\ldots, 2N$. Put $\eps = \min\{\delta_0,\delta_1,\dots,\delta_{N}\}$. We claim that every $m=1,2,\ldots$ we have that
\begin{eqnarray}
&\inf_{i\geq 0} d(f^i(z_{m}),f^{i+m}(z_{m}))\geq \eps.\label{c:***}
\end{eqnarray}
For any $m\leq N$ the condition \eqref{c:***} is satisfied just by the definition of $\eps$, so let us fix $m>N$ and
assume on the contrary that there exists $i_0 \in \mathbb{N}$ such that $d(f^{i_0}(z_{m}),f^{i_0+m}(z_{m}))<\eps$.
If we take any integer $k\geq i_0$ such that $k-i_0<2N$ then by the definition of $\delta_0>\eps$ we have that
$d(f^s(z_{m}),f^{s+m}(z_{m}))<\eta/3$. There is $n\geq 0$ such that $m=n+N$ and there is also $k\geq i_0$
such that $k-i_0<2N$ and $k\not\in \bigcup_{j=0}^\infty [j(n+N)+n,(j+1)(n+N)]$. But then $d(f^{k}(z_{n+N}),f^{n+N+k}(z_{n+N}))<\frac{\eta}{3}$
while by \eqref{c:iv*} we have that
$$
d(f^{k}(z_{n+N}),f^{n+N+k}(z_{n+N}))\geq d(p,q)-\frac{\eta}{3}-\frac{\eta}{3}>\frac{\eta}{3}
$$
which is a contradiction. The proof is completed.
\end{proof}

\begin{lem}\label{lem:Qmn}
Let $(X,d)$ be a compact metric space with $\#X>1$ and let $f\colon X \ra X$ be a surjective continuous map with specification property.
Fix any $p,q\in X$ and any $n,m\geq 0$. The set $Q_m^n(p,q)\subset X\times X$ consisting of pairs $(x,y)$ such that:
\begin{enumerate}
\item there is $l>m$ such that $d(f^{i+l+n}(x),f^{i+l}(p))< \frac{1}{m}$ and $d(f^{i+l}(y),f^{i+l}(q))< \frac{1}{m}$ for $i=0,1,\ldots, 2^l$
\item there is $s>m$ such that $d(f^{i+s}(x),f^{i+s}(p))< \frac{1}{m}$ and $d(f^{i+s+n}(y),f^{i+s}(q))< \frac{1}{m}$ for $i=0,1,\ldots, 2^s$
\end{enumerate}
is open and dense in $X\times X$.
\end{lem}
\begin{proof}
Since $f$ is uniformly continuous, we immediately have that $Q_m^n(p,q)$ is open. Its density follows by the specification property and Remark~\ref{rem8}.
\end{proof}

\begin{thm}\label{thm:spec}
Let $(X,d)$ be a compact metric space with $\#X>1$, let $f\colon X \ra X$ be a surjective continuous map with specification property and assume that $f(p)=p$ for some $p\in X$.
Then there exists $\eps>0$ and a dense Mycielski distributionally $\eps$-scrambled set $D$ such that $f(D)\subset D$ and $p\in D \subset (\Tran(f)\cap \Rec(f))\cup \set{p}$.
\end{thm}
\begin{proof}
Let $\{z_n\}_{n=1}^{\infty}$ and $\eta>0$ be provided by Lemma~\ref{lem:dist_n}, that is
$\inf_{n>0}\inf_{i\geq 0} d(f^{i}(z_n),f^{i+n}(z_n))\geq \eps$.
Denote $\eps=\min\set{\eta/2, \dist(p, \orbp(z_1))/2}$,
put $z_0=p$ and denote:
$$
Q=\bigcap_{n,m\geq 0}Q_m^n(z_n,z_n) \cap Q_m^n(z_0,z_0) \cap Q_m^n(z_{n+1},f(z_{n+1})).
$$
We claim that if $S$ is an infinite set such that $S\times S \subset Q \cup \Delta$ and $u,v\in D=\bigcup_{i=0}^\infty f^i(S)$ are distinct then points $u,v$ form a DC1 pair.
For the proof of this claim, fix any distinct $u,v\in D$. There are $x,y\in S$ (not necessarily distinct) and integers $N,M\geq 0$ such that
 $u = f^N(x)$ and $v = f^M(y)$.

First assume that $x\neq y$ and $N\neq M$. Fix any $m>\frac{4}{\eta}$. Since $(x,y) \in Q^{|N-M|}_m(z_{|N-M|},z_{|N-M|})$, there exist $l,s>m$ such that:
\begin{align}
d(f^{i+l+|N-M|}(x),f^{i+l}(z_{|N-M|}))&<\frac{1}{m}, \label{ineq1}\\
d(f^{i+l}(y),f^{i+l}(z_{|N-M|}))&<\frac{1}{m},\label{ineq2} \\
d(f^{j+s}(x),f^{j+s}(z_{|N-M|}))&<\frac{1}{m}, \label{ineq3}\\
d(f^{j+s+|N-M|}(y),f^{j+s}(z_{|N-M|}))&<\frac{1}{m}\label{ineq4},
\end{align}
for $i=0,\dots,2^l$ and $j=0,\dots,2^s$ respectively. Note that for $i=0,\dots,2^l$ we have that:
\begin{eqnarray*}
d(f^{i+l+|N-M|}(x),f^{i+l}(y))&>&d(f^{i+l+|N-M|}(z_{|N-M|}),f^{i+l}(z_{|N-M|}))\\
&&\quad\quad -d(f^{i+l+|N-M|}(x),f^{i+l}(z_{|N-M|}))\\&&\quad\quad\quad\quad -d(f^{i+l}(z_{|N-M|}),f^{i+l}(y))\\
&\geq& \eta - \frac{2}{m}> \frac{\eta}{2}>\eps.
\end{eqnarray*}
Let us first assume that $N<M$. Then for any $t=l,\ldots, l+2^l-M$ conditions \eqref{ineq3} and \eqref{ineq4} imply. respectively, that:
\begin{align*}
d(f^{t+N}(x),f^{t+M}(z_{N-M}))&<\frac{1}{m}\\
d(f^{t+M}(y),f^{t+M}(z_{N-M}))&<\frac{1}{m}.
\end{align*}
Combining the above inequalities we receive for $t=l,\dots,l+2^l-M$ that:
$$
d(f^{t+N}(x),f^{t+M}(y))>\eta - \frac{2}{m}>\frac{\eta}{2}>\eps.
$$
But then
\begin{eqnarray*}
\Phi_{uv}(\eps)&\leq &
\liminf_{l \rightarrow \infty}\frac{1}{l+2^l}\sharp\{0 \leq t<l+2^l : d(f^{t}(u),f^{t}(v))<\eps\}\\
&\leq& \liminf_{l \rightarrow \infty}\frac{l-1+M}{l+2^l}=0.
\end{eqnarray*}
The case $N>M$ is symmetric, so repeating almost the same calculations with help of \eqref{ineq1} and \eqref{ineq2} we receive again that $\Phi_{uv}(\eps)=0$.

When $x=y$ then we must have $N\neq M$, say $N<M$, and since $S$ is infinite, there is also $z\in S$ such that $(x,z)\in Q$. In particular \eqref{ineq1} is satisfied,
so by uniform continuity of $f$, taking $m$ sufficiently large, we obtain that for $i=l,\ldots,2^l-M$:
\begin{eqnarray*}
d(f^{i}(u),f^{i+N}(z_{M-N}))&=&d(f^{i+N}(x),f^{i+N}(z_{M-N}))<\frac{\eta}{4}, \\
d(f^{i}(v),f^{i+N+(M-N)}(z_{M-N}))&=&d(f^{i+M}(x),f^{i+N+(M-N)}(z_{M-N}))<\frac{\eta}{4}.
\end{eqnarray*}
But since $d(f^{i+N}(z_{M-N}),f^{i+N+(M-N)}(z_{M-N}))>\eta$ for every $i\geq 0$ we can repeat previous calculations obtaining also in this case that
$\Phi_{uv}(\eps)=0$.

If $N=M$ then the only possibility is that $x\neq y$, since $u\neq v$. Consider the set $Q^{N}_m(z_{N+1},f(z_{N+1}))\supset Q$ which by the definition of $Q$ contains the pair $(x,y)$. Again there exist $l>m$ such that:
\begin{eqnarray}
d(f^{i+l+N}(x),f^{i+l}(z_{N+1}))&<\frac{1}{m} \label{ineq5:*}\\
d(f^{i+l}(y),f^{i+l+1}(z_{N+1}))&<\frac{1}{m}\label{ineq6:*},
\end{eqnarray}
for $i=0,\dots,2^l$. Note that when $i<2^l-N$ then we can substitute $i$ with $i+N$ in \eqref{ineq6:*}
obtaining
\begin{eqnarray}
d(f^{i+l+N}(y),f^{i+l+N+1}(z_{N+1}))&<\frac{1}{m}\label{ineq6:**}.
\end{eqnarray}
Therefore, for $j=l,\ldots, l+2^l-N$, combining \eqref{ineq5:*} and \eqref{ineq6:**}, we have that:
\begin{eqnarray*}
d(f^{j+N}(x),f^{j+N}(y))&\geq& d(f^{j}(z_{N+1}),f^{j+N+1}(z_{N+1}))-d(f^{j}(z_{N+1}),f^{j+N}(x))\\
&&\quad\quad-d(f^{j+N}(y),f^{j+N+1}(z_{N+1}))\\
&\geq& \eta - \frac{2}{m}>\frac{\eta}{2}>\eps.
\end{eqnarray*}
By Lemma \ref{lem:dist_n} we have that $d(f^k(z_{N+1}),f^{k+N+1}(z_{N+1}))>\eta$ for all $k>0$, so using the properties of the set $Q^{N+1}_m(z_{N+1},z_{N+1})$ we get the following:
$$
d(f^{\iota+N+1}(x),f^{\iota+N+1}(y))>\eta-\frac{2}{m}>\frac{\eta}{2}
$$
for $\iota=l,\dots,2^l+l-N-1$ provided that $m$ is sufficiently large.
The above inequality leads to $\Phi_{uv}(\eps)=0$ also in this last case.
We have just proved that $\Phi_{uv}(\eps)=0$ for every distinct $u,v\in D$.

It remains to show that for all $t>0$ we have $\Phi^*_{uv}(t)=1$.
Fix any $\xi>0$ and let $\delta>0$ be such that if $d(w,r)<\delta$ then $d(f^j(w),f^j(r))<\xi$ for $j=1,\dots,K$, where $K = N+M+1$.

Let us first consider the case when $x\neq y$. Fix any $m>2/\delta$ and observe that by the definition of $Q$ we have that $(x,y) \in Q^{|N-M|}_m(z_0,z_0)$, hence there exist $l>m$ such that for $i=0,\dots,2^l$ we have:
\begin{align}
d(f^{i+l+|N-M|}(x),z_0)&<\frac{1}{m} \label{ineq5} \\
d(f^{i+l}(y),z_0)&<\frac{1}{m} \label{ineq6},
\end{align}
In particular, for any $t=l,\dots,l+2^l$ we obtain that:
\begin{equation}
d(f^{t+|N-M|}(x),f^{t}(y))<\frac{2}{m}<\delta.\label{ineq7:*}
\end{equation}
Assume first that $N\geq M$. Then by \eqref{ineq7:*} and the choice of $\delta$ and $K$ we obtain that for $t=l,\dots,l+2^l$:
$$
d(f^{t+N}(x),f^{t+M}(y)) = d(f^{t}(u),f^{t}(v)) <\xi.
$$
It immediately implies that:
\begin{eqnarray*}
\Phi^*_{uv}(\xi)&\geq& \limsup_{l \rightarrow \infty}\frac{1}{l+2^l}\sharp\{0\leq j <l+2^l:d(f^{j}(u),f^{j}(v))<\xi \}\\
&\geq& \lim_{l \rightarrow \infty}\frac{2^l}{2^l+l}=1.
\end{eqnarray*}
In the case $N<M$ the calculations are almost identical, thus left to the reader.

In last the case $x=y$ calculations are very similar. First, in such a case we must have $N\neq M$, since $u\neq v$.
Next, since $z_0$ is a fixed point, there is $\tau<\frac{1}{m}$ such that if $d(w,z_0)<\delta$ then $d(f^{|N-M|}(w),z_0)<\frac{1}{m}$.
Fix $m'$ such that $1/m'<\tau$. Since $S$ is infinite, there is $z$ such $(y,z)=(x,z)\in Q^{|N-M|}_{m'}(z_0,z_0)$, and so there is $l>m'>m$
such that for $i=0,\ldots, 2^l$
\begin{equation}
d(f^{i+l}(x),z_0) = d(f^{i+l}(y),z_0)<\frac{1}{m'}<\tau<\frac{1}{m} \label{ineq8:*},
\end{equation}
hence, by the choice of $\tau$, we also have that
\begin{equation}
d(f^{i+l+|N-M|}(y),z_0) <\frac{1}{m} \label{ineq8:**}.
\end{equation}
Now, using \eqref{ineq8:*} and \eqref{ineq8:**} the same way as before we did with \eqref{ineq5} and \eqref{ineq6}
once again we obtain that $\Phi^*_{uv}(\xi)=1$. This completes all possible cases.
showing that indeed the claim holds, that is any two distinct points $u,v\in D$
satisfy that
$$
\Phi_{uv}(\eps)=0 \quad \text{ and }\quad \Phi^*_{uv}(\xi)=1 \quad \text{ for every }\xi>0.
$$

Now we are ready to prove the theorem.
Note that $Q$ is residual by Lemma~\ref{lem:Qmn} and that
relations $\Tran(f)$ and $\Rec(f)$ are residual since $f$ is transitive.
Additionally observe that the relation $R=\bigcap_{n,m}Q_m^n(z_1,z_0) \cap Q_m^n(z_0,z_0)$ is residual. Additionally, if $(x,y)\in R$ then the pair $(x,p)$ is DC1
since for every $m,n$ there is $l>m$ such that for $i=0,\ldots, 2^l$:
\begin{eqnarray*}
d(f^{i+l+1}(x),f^{i+l}(z_{1}))&<\frac{1}{m}\\
\end{eqnarray*}
But since $\dist(z_0, \overline{\orb^+(z_1)})=\dist(z_0, \set{z_1,f(z_1)})>2\eps$ and $m$ can be arbitrarily large,
we easily get that $\Phi_{f^n(x),p}(\eps)=0$ for $n=0,1,\ldots$. Similarly, it is not hard to show that $\Phi^*_{f^n(x),p}(\xi)=1$
for every $\xi>0$

We can apply Theorem~\ref{thm:RelMycielski} to residual relation
$$
Q\cap R\cap (\Tran(f)\times \Tran(f)) \cap (\Rec(f)\times \Rec(f))
$$
obtaining a dense Mycielski set $S\subset X$. But then it is not hard to construct a dense Mycielski set $\tilde{S}$
such that $p\in \tilde{S}\subset S\cup \set{p}$. Denote $D=\bigcup_{i=0}^\infty f^i(\tilde{S})$ and observe that by the definition of
$S$ we have that $(D\setminus \set{p})\times (D\setminus \set{p})\subset (Q\cap R)\cup \Delta$.
Then by properties of relation $Q$ investigated in the first part of the proof we obtain that for any $u,v\in D\setminus \set{p}$
$$
\Phi_{uv}(\eps)=0 \quad \text{ and }\quad \Phi^*_{uv}(\xi)=1 \quad \text{ for every }\xi>0.
$$
But if $u\neq v=p$ then by the properties of relation $R$ we obtain that
$$
\Phi_{up}(\eps)=0 \quad \text{ and }\quad \Phi^*_{up}(\xi)=1 \quad \text{ for every }\xi>0.
$$
Indeed $D$ is distributionally $\eps$-scrambled set. Just by the definition we have that $f(D)\subset D$ and $D\subset (\Tran(f)\cap \Rec(f))\cup \set{p}$
which ends the proof.
\end{proof}

\begin{thm}
\label{thm:myc-distr}
Let $Y$ be an infinite mixing sofic shift
with a fixed point $p\in Y$ and let $\pi\colon (X,f)\ra (Y,\sigma)$ be a factor map such that that $\pi^{-1}(p)$ is finite and consist of periodic points. Then for some $\eps>0$ there is a Cantor distributionally $\eps$-scrambled set $C$ for $f$.

If additionally $\pi^{-1}(p)$ consists of fixed points then there is an invariant Mycielski distributionally $\eps$-scrambled set $M$ for $f$.
\end{thm}
\begin{proof}
Next, note that $\pi$ is a factor map between $(X,f^n)$ and $(Y,\sigma^n)$ for any $m>0$ and $(Y,\sigma^m)$ remains mixing sofic shift as well as $p$ remains fixed for $\sigma^m$ too.
If $(Z,\sigma)$ is an infinite mixing sofic shift then there is a presentation $G$ of $Z$ such that every vertex in $G$ is synchronizing, that is, for a given vertex $J$ in $G$
there is word $v_J$ such that any path presenting $v_J$ have to finish in $J$ (e.g. see Lemma 3.315 and Proposition 3.3.16 in \cite{LindMarc}).
But using the above fact together with assumption that $(Z,\sigma)$ is mixing we see that there is $n$ such that any two vertices in $G$ can be connected by a path of length $n$.
In other words, for any two words $u,v$ allowed for $(Z,\sigma)$ there is a word $w$ of length $n$ such that $uwv$ is also allowed for $Z$ (can be presented by a path on $G$).
It is easy to verify that this condition is equivalent to specification property.

Therefore, we can apply Theorem~\ref{thm:spec} and next Corollary~\ref{cor:InvDC} to the sofic shift $(Y,\sigma^m)$, where $m$ is a common period of periodic points in $\pi^{-1}(p)$.
It is well known that every distributionally $\eps$-scrambled set for $f^n$ is a distributionally $\gamma$-scrambled for $f$, where $\gamma$ depends only on $\eps$ and $n$, which proves the first part of theorem.
Obviously, if $n=1$ then the second part of theorem follows as well.
\end{proof}

\section{Applications}\label{sec:appl}

In this section we provide a possible application of Theorem~\ref{thm:myc-distr}. Example developed in Theorem~\ref{thm:zast} is an extension of
examples of this type considered so far in the literature, e.g. in \cite{OpWil6,OpWil7}. The main difficulty here is
that $\pi^{-1}(p)$, where $p$ is a fixed point in a subshift, contains more than two fixed points for the Poincar\'{e} map.
A mathematically complete and fully rigorous proof of the following Theorem is quite technical and so we decided not to present it in full detail.
We only sketch the calculations that should be performed together with main arguments that follow from these calculations.
The details are left to the reader, who should be able to perform them, since in many aspects methodology is similar to the one presented by authors in \cite{OpWil7}.

\begin{thm}
 \label{thm:zast}
Let the inequalities
\begin{eqnarray}
\label{ineq:kappa}
0<&\kappa&\leq 0.037,\\
\label{ineq:N}
0\leq& N&\leq 0.001
\end{eqnarray}
be satisfied. Then there exists $\eps>0$ such that there is an invariant Mycielski distributionally $\eps$-scrambled set
for the Poincar\'{e} map $\varphi_{(0,T)}$ of the local process generated by the equation
\begin{equation}
\label{eq:glow}
\dot{z}=v(t,z)=\left(1+e^{i\kappa t}\abs{z}^2\right)\z^3 -N,
\end{equation}
where
\begin{equation}
\label{eq:T}
T=\frac{2\pi}{\kappa}.
\end{equation}
\end{thm}

\begin{proof}[Sketch of the proof]
The technical details of the proof are similar to the ones from \cite{OpWil7}.

The proof consists of the two steps.
Firstly we contruct $I$ an invariant subset of $\mathbb{C}$, sofic shift $\Pi$ and a semiconjugacy $g|_I:I\longrightarrow \Pi$ between $\left(I,\varphi_{(0,T)}\right)$ and $(\Pi,\sigma)$. In the second step we investigate properties of the semiconjugacy and then apply Theorem \ref{thm:myc-distr}.

\textbf{Step I.}
Our aim is to construct two appropriate periodic isolating seg\-m\-ents and then apply Theorem~\ref{thm:cont} and \cite[Theorem~7]{WojZgli4}.

We fix $\kappa$ and $N$ satisfying \eqref{ineq:kappa}, \eqref{ineq:N} and set $r=0.568$, $R=1.5$, $\d=10.6$. Let $\varphi^\lambda$ be the local process generated by \eqref{eq:glow} where $\lambda = N$. If there is no confusion, we simply write $\varphi$ instead of $\varphi^\lambda$. Note that  $\{\varphi^\lambda\}_{\lambda\in [0,0.001]}$ is a continuous family of local processes, i.e. close processes have close values for the same starting point $(\sigma,x)$
and sufficiently small time shift $t$.

\picmm{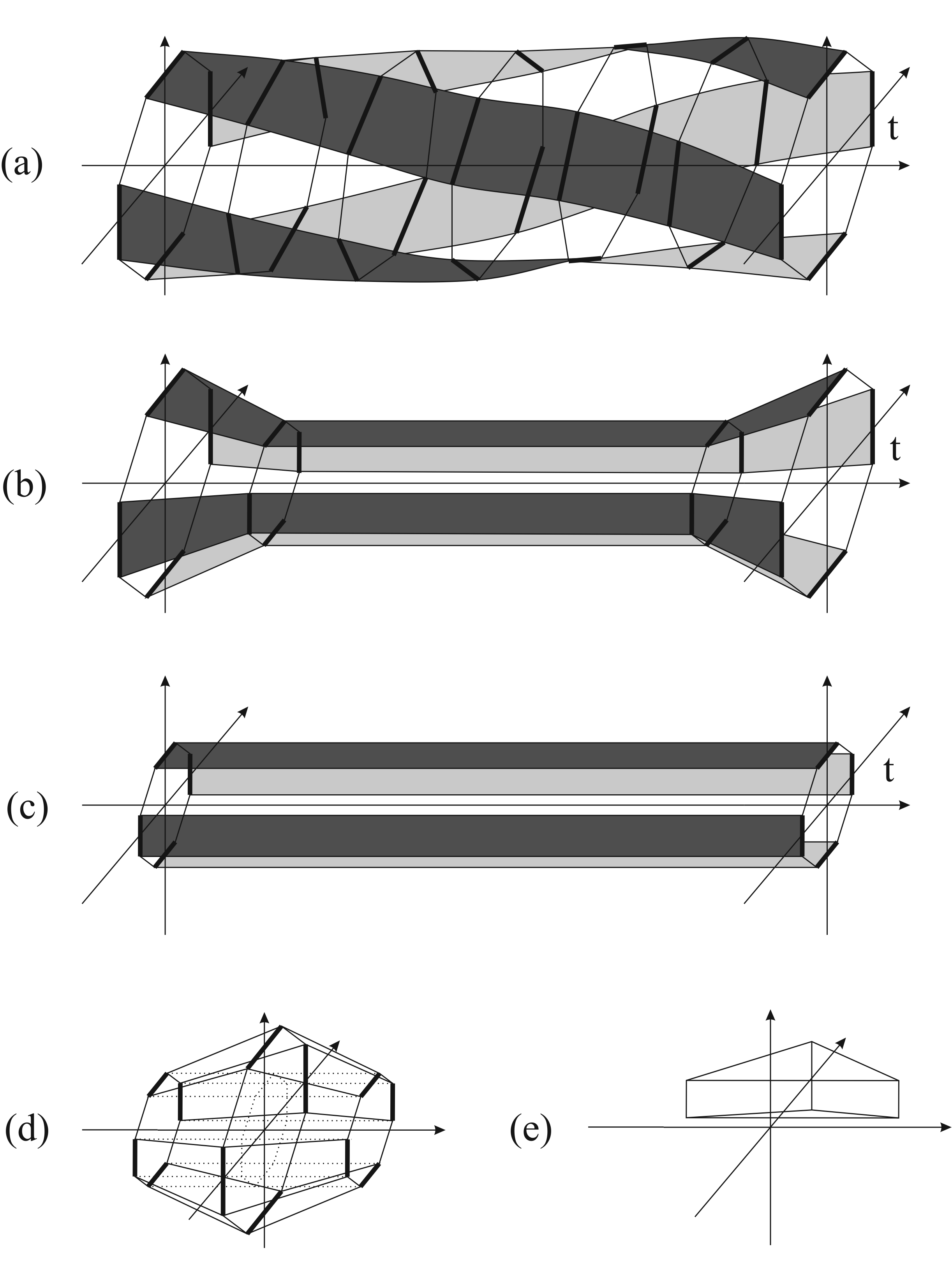}{86mm}{Isolating segments (a)~$W$, (b)~$U$, (c)~$V(\xi)$ and sets  (d)~$Z$ and  (e)~$\widehat{Z}$. Sets $W^{--}$, $U^{--}$, $V(\xi)^{--}$ are marked in grey.}{pic:WUVZ}

We define sets
\begin{align*}
\widehat{B}&=\left\{z\in \mathbb{C}:\abs{\Arg z}\leq \frac{\pi}{8}, \re(z) \leq 1\right\}\cup \{0\},\\
\widehat{B}^\flat&=\left\{z\in\widehat{B}: \re(z) =1\right\}
\end{align*}
and a rotation $h:\mathbb{C}\ni z\mapsto e^{i\frac{\pi}{4}}\in \mathbb{C}$, where as usual $\re(z)$ and $\im(z)$ denote, respectively, real and imaginary part of $z\in \mathbb{C}$. We write
\begin{equation*}
B=\bigcup_{k=0}^{7}h^k\left(\widehat{B}\right), \quad B^\flat= \bigcup_{k=0}^{2} h^{2k} \left(\widehat{B}^\flat\right), \quad B^\sharp= h\left(B^\flat\right)
\end{equation*}
and for each $s\in \R^+$ we define
\begin{equation*}
B(s)=\{sz:z\in B\}, \quad B(s)^\flat=\{sz:z\in B^\flat\}, \quad B(s)^\sharp=\{sz:z\in B^\sharp\}.
\end{equation*}
Note that $B(s)$ is a regular octagon centered at the origin with diameter $2s\tan (\frac{\pi}{8})$. Since $\d=10.6<\frac{T}{2}$, if we set $\omega=\frac{R-r}{\d}$ then the function  $s:[0,T]\longrightarrow [r,R]$ given by the following formula is well defined and continuous:
\begin{equation}
\label{eq:s}
s(t)=\begin{cases}
R-\omega t, & \text{for } t\in [0,\d],\\
r, & \text{for } t\in [\d,T-\d],\\
R-\omega(T-t), &\text{for } t\in [T-\d,T].
\end{cases}
\end{equation}

By a straightforward calculation (inner product of the vector field $(1,v)$ and an outward normal vector at every point of the boundary of the set) it can be easily verified that the sets
\begin{align}
\label{eq:W}
W(\kappa)=&\left\{(t,z)\in [0,T]\times \mathbb{C}:e^{-i\frac{t\kappa}{4}}z\in B(R)\right\},\\
\label{eq:U}
U(\kappa)=&\left\{(t,z)\in [0,T]\times \mathbb{C}:z\in B(s(t))\right\}
\end{align}
are $T$-periodic isolating segments for every $\varphi^\lambda$ and, moreover,
\begin{align}
\label{eq:widW}
W(\kappa)^{--}=&\left\{(t,z)\in [0,T]\times \mathbb{C}:e^{-i\frac{t\kappa}{4}}z\in B(R)^\flat\right\},\\
\label{eq:widwidW}
W(\kappa)^{++}=&\left\{(t,z)\in [0,T]\times \mathbb{C}:e^{-i\frac{t\kappa}{4}}z\in B(R)^\sharp\right\},\\
\label{eq:widU}
U(\kappa)^{--}=&\left\{(t,z)\in [0,T]\times \mathbb{C}:z\in B(s(t))^\flat\right\},\\
\label{eq:widwidU}
U(\kappa)^{++}=&\left\{(t,z)\in [0,T]\times \mathbb{C}:z\in B(s(t))^\sharp\right\}
\end{align}
hold. Note that diameter of $U(\kappa)$ decreases fast enough to satisfy
$U(\kappa)\subset W(\kappa)$ despite the fact that $W(\kappa)$ is a rotation of starting $W(0)$, e.g. see
Figure~\ref{pic:2}.

When number $\kappa$ is fixed, and it causes no confusion, we simply write $W$, $U$ instead of $W(\kappa)$, $U(\kappa)$, respectively, see Figure~\ref{pic:WUVZ}(a), \ref{pic:WUVZ}(b).

\picmm{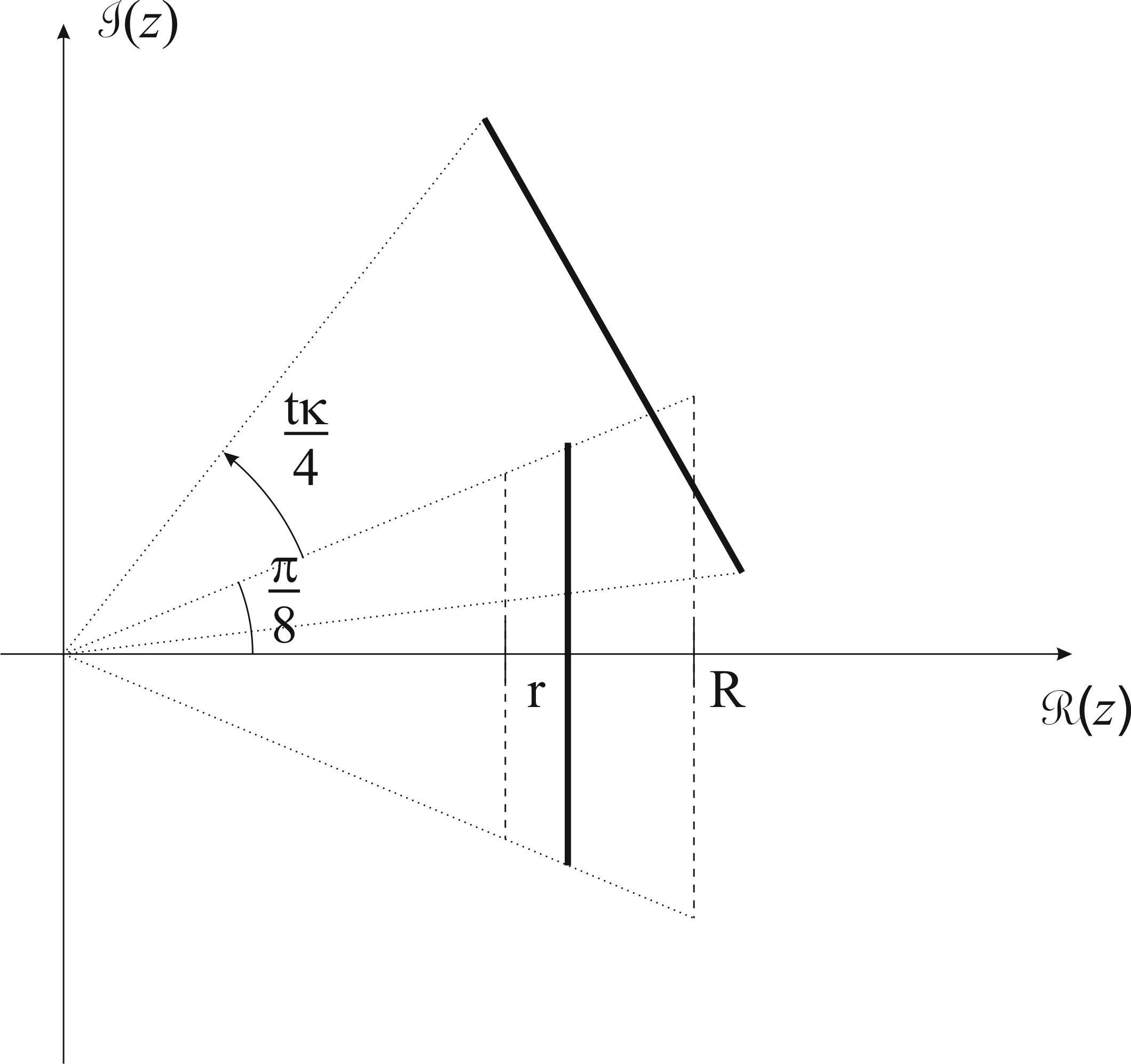}{74mm}{Segment $U(\kappa)$ narrows when octagon $W(\kappa)$ is rotating.}{pic:2}

It can also be verified by appropriate calculations that sets $W(\kappa)$ and $U(\kappa)$ satisfy conditions (G1) and (G2) uniformly with respect to $\lambda$. Moreover, $U(\kappa)$ can be decomposed into four connected components  $U(\kappa)^{--}=E[1]\cup E[2]\cup E[3]\cup E[4]$, where
\begin{align*}
E[1]=&\left\{(t,z)\in U(\kappa)^{--}: \abs{\Arg z}\leq \frac{\pi}{8}\right\},\\
E[2]=&\left\{(t,z)\in U(\kappa)^{--}: (t,h^{-2}(z))\in E[1]\right\},\\
E[3]=&\left\{(t,z)\in U(\kappa)^{--}: (t,h^{-4}(z))\in E[1]\right\},\\
E[4]=&\left\{(t,z)\in U(\kappa)^{--}: (t,h^{-6}(z))\in E[1]\right\}.
\end{align*}

In order to define desired semi-conjugacy, first we define a subshift $\Pi\subset \Sigma_5$ such that $c\in \Pi$ if and only if the following rules are satisfied:
\begin{itemize}
\item[(M1)] if $c_i=k$ for some $k\in \{1,2,3,4\}$, then $c_{i+1}\in \{0,k, (k \mod 4)+1\}$,
\item[(M2)] if $c_p=0$ for all $p\leq i$, then $c_{i+1}\in \{0,1,2,3,4\}$,
\item[(M3)] if $c_j=0$ for all $j\in \{p+1,\ldots,i\}$ and $c_p=k\neq 0$, then $c_{i+1}\in \{0,k, (k \mod 4)+1\}$.
\end{itemize}
Note that conditions (M1)--(M3) in fact define a sofic shift $\Pi$ whose presentation is given on Figure~\ref{pic:sofic}.

\picmm{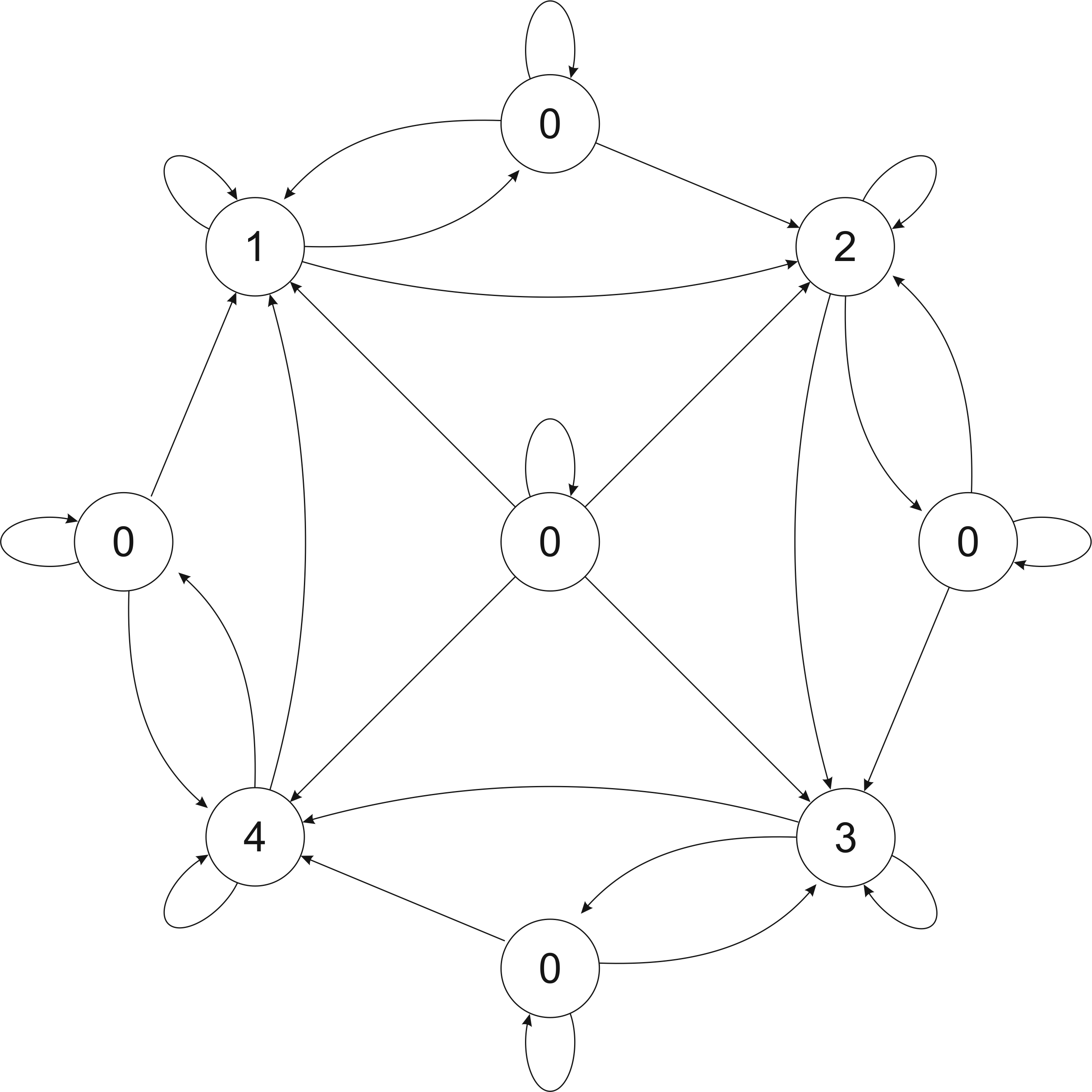}{73mm}{Presentation of the sofic shift $\Pi$.}{pic:sofic}

A set of forbidden words
for $\Pi$ is presented below
\begin{equation*}
\set{10^k3, 10^k4, 20^k1, 20^k4, 30^k2, 30^k1, 40^k2, 40^k3: k\in \N\cup\{0\}}
\end{equation*}
Observe that shift $\Pi$ is strictly sofic, that is $\Pi$ is not a shift of finite type. Indeed, if $\Pi$ is of finite type then there is $n>0$ such that if $uv$ and $vw$ are allowed words then $uvw$ is also allowed, provided that $|v|\geq n$, i.e. $v$ has at least $n$ symbols (see \cite{JonMar} for more comments).
In our case, in particular, $10^n$ and $0^n3$ are allowed for every $n>0$ but $10^n 3$ is always forbidden.

Since the presentation of the shift $\Pi$ is an irreducible graph and it has self loop, shift $\Pi$ is mixing (see \cite[Prop. 4.5.10.]{LindMarc}), so we are in position to apply Theorem~\ref{thm:myc-distr}, provided that we can semi-conjugate Poincar\'{e} map on an invariant subset with $(\Pi,\sigma)$.

Let $c\in \Pi$ be an $n$-periodic point of $\sigma$.
Then, by Theorem \ref{thm:cont}, the equality
\begin{equation*}
{\mathrm{ind}}\left(\varphi^0_{(0,nT)}\mid_{(W_0\setminus W_0^{--})_c^0}\right) ={\mathrm{ind}}\left(\varphi^\lambda_{(0,nT)}\mid_{(W_0\setminus W_0^{--})_c^\lambda}\right)
\end{equation*}
is satisfied, while following the proof of Theorem 7 in \cite{WojZgli4}, we can prove that ${\mathrm{ind}}\left(\varphi^0_{(0,nT)}\mid_{(W_0\setminus W_0^{--})_c^0}\right)\neq 0$ and finally that
\begin{equation}
\label{eq:index}
{\mathrm{ind}}\left(\varphi^\lambda_{(0,nT)}\mid_{(W_0\setminus W_0^{--})_c^\lambda}\right)\neq 0
\end{equation}
holds.
By \eqref{eq:index} and Theorem~\ref{tws} there exists a point $z\in (W_0\setminus W_0^{--})_c^\lambda$ such that $\varphi^\lambda_{(0,nT)}(z)=z$.
Now we fix $\lambda$ and simply write $\varphi$ instead of $\varphi^\lambda$.

Let us denote
\begin{equation}
\widetilde{I}=\set{z\in W_0:\varphi_{(0,t+kT)}(z)\in W_t \text{ for } t\in [0,T], k\in \Z}.
\end{equation}
and define a (continuous) map $g:\widetilde{I}\longrightarrow \Sigma_5$ by
\begin{equation*}
g(z)_l=\begin{cases}
0, & \text{if } \varphi_{(0,t+lT)}(z)\in U_t \text{ for all } t\in (0,T),\\
k, & \text{if } \varphi_{(0,lT)}(z) \text{ leaves } U \text{ in time less than } T \text{ through } E[k].
\end{cases}
\end{equation*}
Let us notice that, by the same argument as used in \cite{WojZgli4}, for a given $l\in \Z$ the trajectory $\varphi_{(0,lT)}(z)$ can leave $U$ in time less than $T$ at most once, so $g$ is well defined.

Directly from the definition of $g$ we obtain that $\s\circ g=g\circ\varphi_{(0,T)}$. Moreover, by \eqref{eq:index} and Theorem~\ref{tws}, if $c\in \Pi$ is $n$-periodic, then $g^{-1}(c)$ contains an $n$-periodic point of $\varphi_{(0,T)}$. Thus $\Pi\subset g\left(\widetilde{I}\right)$. Let $I=g^{-1}\left(\Pi\right)$. Then $g\mid_I:I\longrightarrow \Pi$ is a semiconjugacy between $\varphi_{(0,T)}\mid_I$ and $\s\mid_\Pi$.


\textbf{Step II.}
To apply Theorem \ref{thm:myc-distr} and finish the proof we only need to show that
\begin{equation}
\label{ineq:g-1}
\#g^{-1}(\{0^\infty\})\leq 3
\end{equation}
and all the points in $g^{-1}(\{0^\infty\})$ are fixed points of $\varphi_{(0,T)}$.

Condition \eqref{ineq:g-1} equivalently means that there are at most three solutions $z:\R\to \mathbb{C}$ of \eqref{eq:glow} which stay in the set $U^{\infty}$, which in other words mean that $z(t+nT)\in U_t$ for every $t\in [0,T]$ and $n\in \Z$.

\emph{Case I.} Let $N=0$. Then $\psi\equiv 0$ is a solution of \eqref{eq:glow}. We show that it is the only solution in $U^\infty$. We start with defining set
\begin{equation}
\label{eq:V}
V(\xi)=[0,T]\times B(\xi)
\end{equation}
where $\xi>0$.

It can be observed that for every $\xi\in (0,r]$ (see Figure~\ref{pic:WUVZ}(c)) we have that
\begin{equation}
\label{eq:V--}
V(\xi)^{--}=[0,T]\times B(\xi)^\flat, \quad V(\xi)^{++}=[0,T]\times B(\xi)^\sharp
\end{equation}
hold, i.e. $V(\xi)$ is an isolating segment. Thus, by \cite[Lemma 3.4]{OpWil7}, $\psi$ is the only solution which stays in $V(r)^{\infty}$. By direct calculations, every solution which enters $(U\setminus V(r))^\infty$ has to leave $U^\infty$. Thus condition \eqref{ineq:g-1} is satisfied.

\emph{Case II.} Now, let $N>0$ be satisfied. We write
\begin{align*}
M=&\sqrt[3]{N}\\
P_{\langle k \rangle}=& e^{\frac{2}{3}\pi i k} \text{ for } k\in\set{0,1,2},\\
j: \mathbb{C}^2\ni & (P,z) \mapsto j(P,z) = \overline{P}z+MP\in \mathbb{C}.
\end{align*}
For $\eta >0$ we define sets
\begin{align*}
D(\eta)=& \left\{ z\in\mathbb{C}: \abs{\re\left[z \right]}\leq \eta, \abs{\im\left[z \right]}\leq \eta \right\},\\
D(\eta)^\sharp=& \left\{ z\in D(\eta): \abs{\re\left[z \right]}= \eta \right\},\\
D(\eta)^\flat=& \left\{ z\in D(\eta):  \abs{\im\left[z \right]}= \eta \right\},\\
E_{\langle k \rangle}(\eta) = & [0,T]\times j(\set{P_{\langle k \rangle}}\times D(\eta)) \text{ for } k\in\set{0,1,2}.
\end{align*}

Similarly as before it can be shown by appropriate analysis of the vector field that for every $\eta\in[0.004 M, 0.383 M]$ and every $k\in\set{0,1,2}$ each of the sets $E_{\langle k \rangle}(\eta)$ is an isolating segment such that
\begin{align*}
E_{\langle k \rangle}(\eta)^{--} = & [0,T]\times j(\set{P_{\langle k \rangle}}\times D(\eta)^\sharp),\\
E_{\langle k \rangle}(\eta)^{++} = & [0,T]\times j(\set{P_{\langle k \rangle}}\times D(\eta)\flat)
\end{align*}
and moreover
\begin{align*}
{\mathrm{Lef}}\left(\mu_{E_{\langle k \rangle}\left(\eta\right)}\right)= \chi\left(E_{\langle k \rangle}\left(\eta\right)_0\right)-\chi\left(E_{\langle k \rangle}\left(\eta\right)_0^{--}\right)=1-2=-1\neq 0,
\end{align*}
where $\chi$ denote the Euler characteristic. By Theorem~\ref{tws}, there exists $T$-periodic solution $\psi_k$  of the equation~\eqref{eq:glow} inside the set $E_{\langle k \rangle}\left(0.004 M\right)^\infty$. Note that $\psi_k(0)\in g^{-1}(\{0^\infty\})$ and we can also find the following upper bound
\begin{equation}
\label{ineq:oszpsi1}
\abs{\psi_k}(t)\leq 0.004 \sqrt{2} M\quad\quad \text{ for every } t.
\end{equation}

To finish the proof it is enough to show that $\psi_0$, $\psi_1$ and $\psi_2$ are the only solutions of \eqref{eq:glow}  contained completely in $U^\infty$. It can be done in five steps:
\begin{enumerate}[(i)]
\item\label{step.i} $\psi_k$ is the only solution in $E_{\langle k \rangle}\left(0.004 M\right)^\infty$ for every $k\in\set{0,1,2}$,
\item\label{step.ii} $\psi_k$ is the only solution in $E_{\langle k \rangle}\left(0.383 M\right)^\infty$ for every $k\in\set{0,1,2}$,
\item\label{step.iii} $\psi_0$, $\psi_1$ and $\psi_2$ are the only solutions in $V\left(1.3M\right)^\infty$,
\item\label{step.iv} $\psi_0$, $\psi_1$ and $\psi_2$ are the only solutions in $V(r)^\infty$,
\item\label{step.v} every solution which enters $(U\setminus V(r))^\infty$ has to leave $U^\infty$,
\end{enumerate}
where $V(\xi)$ is defined by \eqref{eq:V}. Now we are going sketch proofs of conditions \eqref{step.i}--\eqref{step.v}, which will end the proof.

\eqref{step.i}: By the symmetry
\begin{equation}
\label{eq:symmetry}
v(t,z)=v(t,P_{\langle k \rangle} z) \text{ for all } (t,z)\in \R\times \mathbb{C} \text{ and } k\in\set{0,1,2},
\end{equation}
it is enough to prove that $\psi_0$ is the only solution in $E_{\langle 0 \rangle} \left(0.004 M\right)^\infty$. The idea and calculations in the case of $\psi_1$ and  $\psi_2$ are similar.

We make the following change of variables
\begin{equation}
\label{eq:zmiana1}
w=P_{\langle 0 \rangle}(z-MP_{\langle 0 \rangle})-\psi_0
\end{equation}
and show that in a neighborhood of origin term $3M^2\overline{w}$ is dominating in the equation. Since the dynamics of the equation is determined by this term in the ball centered at origin with radius equal to $0.01M\sqrt{2}$, the only solution of the equation which never leaves the neighbourhood is $w\equiv 0$. Since $\left(E_{\langle k \rangle}\left(0.004 M\right)\right)_0$ (in $z$ coordinates) is contained in this neighbourhood (in $w$ coordinates), $\psi_0$ is the only solution in $E_{\langle 0 \rangle}\left(0.004 M\right)^\infty$.

\eqref{step.ii}:  It is a direct consequence of the \cite[Lemma 3.4]{OpWil7}.

\picmm{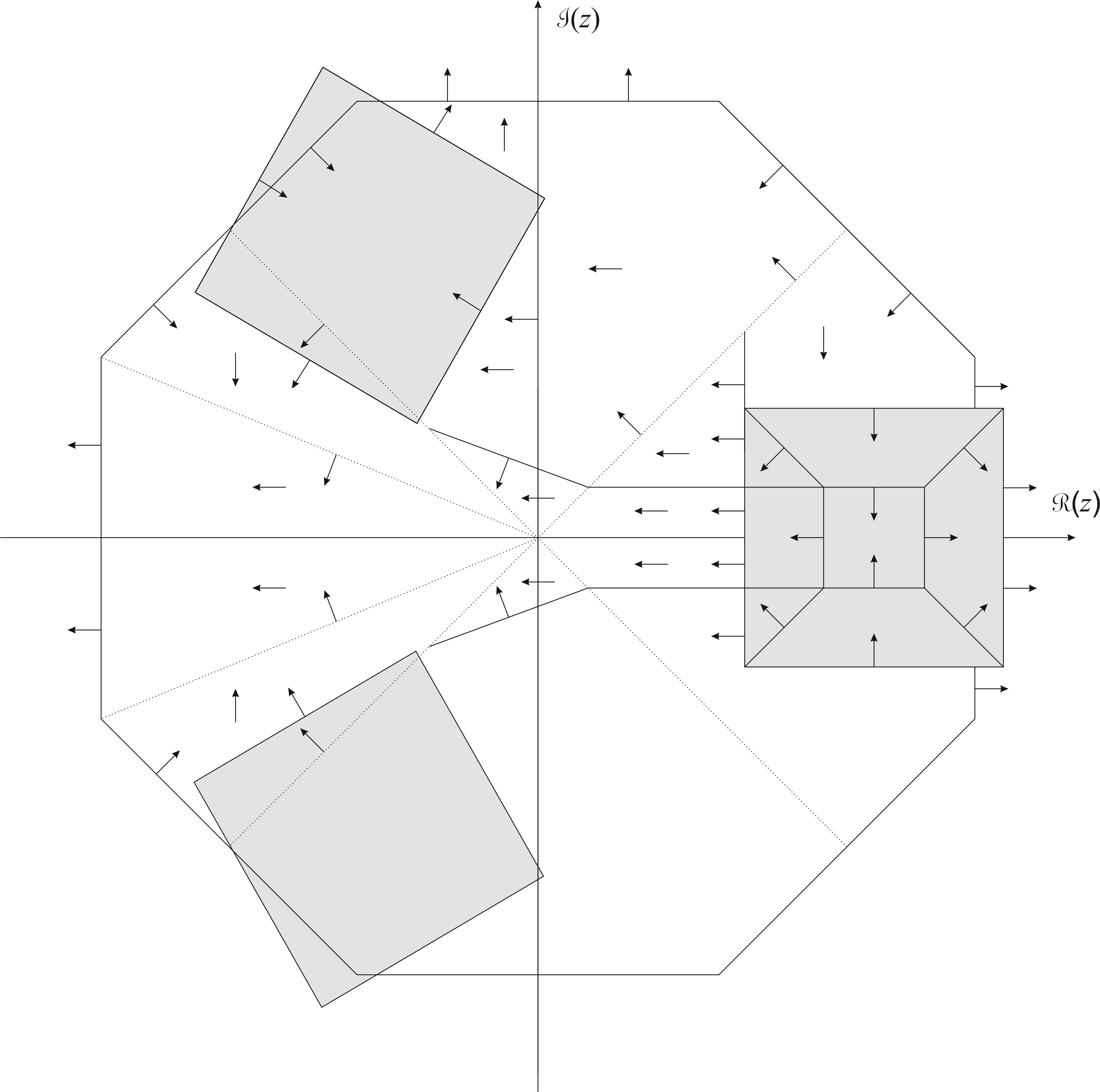}{136mm}{The qualitative behavior of the vector field $v$ inside $\left(V\left(1.3M\right)\right)_0$. Arrows shows general directions of trajectories. If an arrow starts at a line, the line cannot be passed by any trajectory in the opposite direction. Sets $\left(E_{\langle k \rangle}\left(0.004 M\right)\right)_0$ for $k=0,1,2$ are marked in grey.}{pic:glowny}

\eqref{step.iii}: It it enough to show that every solution which is contained for some time $t_0$ in $V\left(1.3M\right)^\infty$ but outside every $E_{\langle k \rangle}\left(0.004 M\right)^\infty$, where $k=0,1,2$, has to leave $V\left(1.3M\right)^\infty$ for some $t_1\neq t_0$. It can be done by a careful analysis of the vector field inside the set $V\left(1.3M\right)$. Since the dominating term of the vector field $v$ is $\z^3-N$, and it does not depend on the time $t$, it is enough to investigate it in the subset of the complex plane $\left(V\left(1.3M\right)\right)_0$ taking only into account perturbations of $\z^3-N$ which come from the other terms of $v$. The qualitative behavior of $v$ is sketched in the Figure~\ref{pic:glowny}.

\eqref{step.iv}: Note that for every $\xi\in [1.3M,r]$ the set $V(\xi)$ is an isolating segment such that \eqref{eq:V--} hold. Now, it is enough to apply \cite[Lemma 3.4]{OpWil7}.

\eqref{step.v}: It can be proved in an analogous way to the proof of \cite[Lemma 3.2]{OpWil7} by analyzing $Z=\left((U\setminus V(r))^\infty\right)_{[-\Delta, \Delta]}$ and its part $\widehat{Z}$ depicted in Figure \ref{pic:WUVZ}(d) and (e), respectively.

\end{proof}

\section*{Acknowledgements}
During 16th Czech-Slovak Workshop on Discrete Dynamical Systems (CSWDDS 2012) held between 11th and 15th of June, 2012 at
Pustevny in Beskydy mountains in  Czech Republic Jana Dole\v{z}elov\'a announced that she can prove a version of Theorem~\ref{thm:spec} with an additional assumption that there are
periodic points of infinitely many different periods. This result was obtained by her completely independently of our research.
While at the time present paper was almost finished, we acknowledge her priority in answering Question~1 in \cite{OprochaDS}.

This research was supported by the Polish Ministry of Science and Higher Education.
Research of P. Oprocha leading to results contained in this paper was supported by the Marie Curie European Reintegration Grant of the European Commission under grant agreement no. PERG08-GA-2010-272297.

The financial support
of these institutions is hereby gratefully acknowledged.

\bibliographystyle{amsplain}
\bibliography{fow_ref}
\end{document}